\documentclass[a4paper,11pt]{article}
\usepackage[margin=1in]{geometry}
\usepackage[utf8]{inputenc} 
\usepackage[T1]{fontenc}    
\usepackage{hyperref}
\hypersetup{
	colorlinks=true,
	linkcolor=blue,
	citecolor=blue,
	filecolor=blue,      
	urlcolor=blue
}     
\usepackage{url}          
\usepackage{booktabs}      
\usepackage{placeins}
\usepackage{amsfonts,amsthm}      
\usepackage{amsmath}
\usepackage{array}
\usepackage{nicefrac}      
\usepackage{microtype}     
\usepackage{lipsum}
\usepackage{graphicx,xcolor}
\usepackage{bm} 
\usepackage{enumitem}
\usepackage{natbib}

\usepackage{lineno}
\graphicspath{ {./images/} }
\newtheorem{theorem}{Theorem}

\newtheorem{remark}{Remark}
\newtheorem{assumption}{Assumption}

\newtheorem{corollary}{Corollary}

\usepackage{algorithm,algorithmic}

\newcommand{\Db}{\mathbf{D}}

\begin{document}
	\title{Fast Operator-Splitting Methods for Nonlinear Elliptic Equations}
    \date{}
\author{
Jingyu Yang\thanks{Department of Mathematics, Hong Kong Baptist University, Kowloon Tong, Hong Kong. Email: {\bf 24481092@life.hkbu.edu.hk}.}
,
Shingyu Leung\thanks{Department of Mathematics, The Hong Kong University of Science and Technology, Clear Water Bay, Hong Kong. Email: {\bf masyleung@ust.hk}.}
,
Jianliang Qian\thanks{Department of Mathematics and Department of CMSE, Michigan State University, East Lansing, MI 48824, USA. Email: {\bf jqian@msu.edu}. }
,
Hao Liu\thanks{Department of Mathematics, Hong Kong Baptist University, Kowloon Tong, Hong Kong. Email: {\bf haoliu@hkbu.edu.hk}.}
}   
	\maketitle
\begin{abstract}
	Nonlinear elliptic problems arise in many fields, including plasma physics, astrophysics, and optimal transport. In this article, we propose a novel operator-splitting/finite element method for solving such problems. We begin by introducing an auxiliary function in a new way for a semilinear elliptic partial differential equation, leading to the development of a convergent operator-splitting/finite element scheme for this equation. The algorithm is then extended to fully nonlinear elliptic equations of the Monge-Amp\`ere type, including the Dirichlet Monge-Amp\`ere equation and Pucci’s equation. This is achieved by reformulating the fully nonlinear equations into forms analogous to the semilinear case, enabling the application of the proposed splitting algorithm. In our implementation, a mixed finite element method is used to approximate both the solution and its Hessian matrix. Numerical experiments show that the proposed method outperforms existing approaches in efficiency and accuracy, and can be readily applied to problems defined on domains with curved boundaries.
\end{abstract}
     

\section{Introduction}
Nonlinear elliptic partial differential equations (PDEs) arise in many fields, including interface problems (the semilinear elliptic PDEs \cite{chen2020two}), optimal mass transportation (the Monge-Amp\`ere type equations  \cite{villani2008optimal}), and segregation of populations with high competition (the Pucci's equation \cite{Caffarelli2}). Designing accurate and efficient numerical methods to solve this class of equations has been an important topic for a long time. In this article, we design a new class of operator-splitting methods for solving fully nonlinear elliptic equations by first reformulating the equation into a form analogous to a semilinear elliptic equation, and then splitting the resulting equation into two equations with the help of an auxiliary function.

To start with, we consider a semilinear elliptic partial differential equation in the following form:
\begin{equation}
	-\Delta u= f(x,u), \  \mbox{in} \  \Omega.
	\label{eq.semilinear}
\end{equation} 
The numerical study of (\ref{eq.semilinear}) has progressed significantly, largely due to its tractable structure in which the nonlinearity depends solely on the solution $u$ and not on its derivatives. Numerous numerical methods have been developed to solve this equation. A two-grid method was introduced in \cite{Xu}, leveraging finite element spaces defined on both coarse and fine grids to efficiently approximate the solution, and an adaptive finite element method based on a multilevel correction scheme was presented in \cite{Lin}. A localized orthogonal decomposition method was studied in \cite{henning2014localized} for semilinear elliptic equations with heterogeneous, variable coefficients. More recently, a ResNet with ReLU$^2$ activations was developed in \cite{chen2024analysis} to solve this equation. All the above-cited methods couple the nonlinear term with the linear one. In our method, we introduce an auxiliary function to decouple the nonlinear term from the linear one so that we can develop a fast and easy-to-implement operator-splitting method to solve the semilinear elliptic equation. 

On the other hand, although existence, uniqueness, and regularity theories of solutions for general fully nonlinear PDEs are well documented in \cite{Caffarelli1}, numerical methods for these equations are more challenging to develop. One reason is that a fully nonlinear PDE generally lacks a variational structure so that the powerful weak formulation designed for linear PDEs is not directly applicable. Since fully nonlinear elliptic PDEs of Monge-Amp\`ere type arise in fields such as optimal transport, antenna reflector design, and mesh deformation \cite{villani2008optimal,  Boonkkamp}, numerical methods for these equations are called for. Different strategies for handling the nonlinearity of the Monge-Amp\`ere operator result in different algorithms. One approach is to reformulate the Monge-Amp\`ere equation as an optimization problem, which is then solved via the augmented Lagrangian method or a (relaxed) least-squares algorithm \cite{Dean1,Dean2,Caboussat1};  this class of methods has been extended to optimal transport problems in \cite{prins2015least,yadav2019monge}. In \cite{two}, the authors proposed a standard finite-difference method based on Gauss-Seidel iterations, and in \cite{Oberman1}, the author proposed a wide stencil finite-difference discretization for the Monge-Amp\`ere equation; more finite-difference algorithms can be found in \cite{neilan2020monge, Froese2, Benamou2, Nochetto1}. In \cite{Neilan1}, the authors utilized a mixed finite element method to solve a fourth-order quasilinear PDE, which in turn approximates the Monge-Amp\`ere equation by the vanishing moment method; see \cite{Neilan1} for more details. In \cite{Wangp}, the authors proposed the spectral collocation method to solve this equation, and in \cite{DG}, the authors proposed the discontinuous Galerkin method to handle the equation. Based on a divergence form of the Monge-Amp\`{e}re operator, in \cite{Liu2,Liu1}, an operator-splitting method is proposed to handle this equation, in which a mixed finite element method with piecewise linear finite elements is used to discretize the equation; the operator-splitting method was further developed to solve the eigenvalue problem of the Monge-Amp\`ere operator and the Minkowski problem in \cite{liu2022efficient,glowinski2020numerical,Liu3}. The operator-splitting method developed in \cite{Liu2,Liu1} needs a projection step to enforce the convexity, which is unnatural. Consequently, in this article, we propose a new operator-splitting strategy to enforce the convexity naturally.   

The Pucci's equation, which involves the Pucci's extremal operator, represents an important class of fully nonlinear PDEs with significant applications in stochastic control theory and population models \cite{Bensoussan, Caffarelli2}. The theoretical aspects of Pucci's equation have been thoroughly investigated in the works of \cite{Felmer2}, and a solid mathematical foundation for this problem has been established. However, only a few numerical methods are available for this equation in the literature. Specifically, three distinct nonlinear least-squares finite-element methods for the two-dimensional Pucci's equation with Dirichlet boundary conditions were proposed in \cite{Glowinski, Brenner, Caboussat2}. Additionally, a non-variational finite element method was employed to study the two-dimensional case in \cite{lakkis2013finite}. Our proposed method can be easily applied to solve Pucci's equation, yielding the optimal convergence rate  in terms of mesh size $h$ used for discretizing continuous piecewise linear finite elements and higher accuracy compared to existing methods.

The operator-splitting methodology is an effective approach to solve complicated problems in imaging, communication, science, and engineering \cite{Splitting}. It decomposes a complicated problem into several easy-to-solve subproblems. In this article, we first propose an efficient operator-splitting method to solve the semilinear elliptic equation of the form (\ref{eq.semilinear}). To do that, we first introduce an auxiliary function to decouple the nonlinear term from the linear term and accordingly convert the problem into a PDE system. The system is then solved by the Lie-splitting method, and we further establish the convergence of the resulting splitting method for the semilinear elliptic equation. With the splitting method for the semi-linear elliptic equation at our disposal, we extend this algorithm to solve fully nonlinear elliptic equations of the Monge-Amp\`ere type. 

Our contributions are summarized as follows:
\begin{enumerate}
	\item We propose an operator-splitting method for semilinear equations of the form (\ref{eq.semilinear}) and establish its convergence.
	\item By utilizing an eigenvalue formulation of Hessian matrices, we extend the method to solve Monge-Amp\`ere type equations, which cover the Dirichlet Monge-Amp\`ere equation and Pucci’s equation.
	\item The spatial discretization employs a mixed finite element method with piecewise linear bases, facilitating straightforward implementation on irregular domains.
	\item Our method achieves optimal convergence rates for problems admitting classical solutions and demonstrates greater efficiency compared to existing methods, while maintaining comparable or improved accuracy.
\end{enumerate}

Our paper is organized as follows: In Section \ref{sec.semilinear}, we propose a novel operator-splitting method for semilinear elliptic PDEs, and we also analyze the convergence of the splitting method. In Section \ref{sec.MA}, we demonstrate how to apply the proposed method to solve Monge-Amp\`ere type equations, including the Dirichlet Monge-Amp\`ere equation and Pucci's equation. 
In Section \ref{sec.FEM}, we use a mixed finite element method with piecewise linear bases to discretize the resulting PDE system due to splitting. Section \ref{sec.implementation} provides a detailed implementation of our algorithm. In Section \ref{sec.experiments}, we demonstrate the effectiveness of the proposed algorithm by carrying out a variety of numerical experiments, and we further compare it to existing methods. Section \ref{sec.conclusion} concludes the paper.

\section{Operator Splitting Method of Semilinear Elliptic Problem}
\label{sec.semilinear}
We introduce some notations. Let $\Omega\subset \mathbb{R}^2$
be a bounded domain, and denote the standard Sobolev space by $W^{m,p}(\Omega)$, which is equipped with the norm and semi-norm defined by  
\begin{align*}
	\|\phi\|_{m,p} = \left( \sum_{|\alpha| \leq m} \|D^\alpha \phi\|_{L^p(\Omega)}^p \right)^{1/p} \ \mbox{and}\ 
	|\phi|_{m,p} = \left( \sum_{|\alpha| = m} \|D^\alpha \phi\|_{L^p(\Omega)}^p\right)^{1/p}, 
\end{align*}
respectively, for $\phi \in W^{m,p}(\Omega)$. As usual, we denote $H^m(\Omega) = W^{m,2}(\Omega)$ with the norm $\|\cdot\|_m = \|\cdot\|_{m,2}$ and semi-norm $|\cdot|_m = |\cdot|_{m,2}$ . When $m = 0$, $H^0(\Omega)=L^2(\Omega)$ and $\|\cdot\|_0 = \|\cdot\|_{0,2}=\|\cdot\|_{L^2}$. The set $H_g^1(\Omega)$ contains functions belonging to $H^1(\Omega)$ with trace $g$ on $\partial \Omega$. 

Furthermore, let $(\cdot, \cdot)$ denote the $L^2$ inner product. If $\phi \in H_0^1(\Omega)$, then it follows from the  Poincar\'{e} inequality \cite{Evans}
\begin{equation}
	\|\phi\|_{0} \leq C_1|\phi|_{1},
	\label{eq.poincare}
\end{equation}
where the constant $C_1>0$ only depends on $\Omega$.

\subsection{Semilinear Elliptic Equation and Its Reformulation}
We consider the following semilinear elliptic equation with the Dirichlet boundary condition:
\begin{equation}\label{s1}
	\begin{cases}
		-\Delta u= f(x,u), \  \mbox{ in } \   \Omega, \\
		u=g, \ \mbox{ on }  \ \partial \Omega.
	\end{cases}
\end{equation}
We propose an operator-splitting method to solve \eqref{s1} by decoupling the nonlinear term $f(x,u)$ from the linear term $-\Delta u$. Specifically, by introducing an auxiliary function $w$, the boundary value problem \eqref{s1} is equivalent to the following PDE system: 
\begin{equation}\label{s2}
	\begin{cases}
		\begin{cases}
			-\Delta u= f(x,w), \ \mbox{ in } \ \Omega,\\
			u=g, \  \mbox{ on } \ \partial \Omega,\\
		\end{cases}\\
		w-u=0,\ \mbox{ in } \ \Omega.
	\end{cases}
\end{equation}
We associate the PDE system \eqref{s2} with the following initial value problem:
\begin{equation}\label{ivp1}
	\begin{cases}
		\begin{cases}
			\dfrac{\partial u}{\partial t} -\Delta u= f(x,w),\  \mbox{ in } \ \Omega\times (0,+\infty),\\
			u=g, \  \mbox{ on }   \ \partial \Omega\times(0,+\infty),
		\end{cases}\\
		\dfrac{\partial w}{\partial t}+\gamma(w-u)=0, \  \mbox{ in } \ \Omega\times (0,+\infty),\\
		u(0)=u_0,w(0)=w_0,
	\end{cases}
\end{equation}
where $\gamma$ is a positive parameter controlling the evolution speed of $w(t)$.
Consequently, solving system \eqref{s2} is reduced to finding the steady state solution of problem \eqref{ivp1}. 

For the choice of $\gamma$, it should be selected so that $w(t)$ evolves with a speed similar to that of $u(t)$. Thus, we suggest taking 
$\gamma=\lambda_0$,
where $\lambda_0$ denotes the smallest of eigenvalue of $-\Delta$ in $H_0^1(\Omega)$. A similar strategy is adopted in \cite{Liu2}.

\subsection{Time Discretization by Operator Splitting}
Problem \eqref{ivp1} is well-suited to be solved by operator-splitting methods. Here we adopt the simple Lie-splitting scheme \cite{Splitting}.

Let $\tau>0$ denote the time step with $t^n=n\tau$, and let $(u^n, w^n)$ represent the numerical solution of $(u,w)$ at $t=t^n$, where $n=0,1,2,\cdots$. Assume that $(u^0,w^0)=(u_0,w_0)$ is given. For $n\geq 0$, the operator splitting scheme updates
$(u^n,w^n)\rightarrow(\hat{u}^{n+1},\hat{w}^{n+1})\rightarrow(u^{n+1},w^{n+1})$ via two substeps:

$\textbf{Substep 1}$: Solve
\begin{equation} \label{step1}
	\begin{cases}
		\begin{cases}
			\dfrac{\partial u}{\partial t} -\Delta u=f(x,w^n) , \   \mbox{ in } \ \Omega\times(t^n, t^{n+1}),\\
			u=g, \    \mbox{ on }  \  \partial\Omega\times(t^n, t^{n+1}),
		\end{cases} \\
		\dfrac{\partial w}{\partial t}=0,\    \mbox{ in }  \ \Omega\times(t^n, t^{n+1}),\\	
		(u(t^n),w(t^n))=(u^n, w^n),
	\end{cases}
\end{equation}
and set 
$  \hat{u}^{n+1}=u(t^{n+1})$ and $ \hat{w}^{n+1}=w(t^{n+1})$.  

$\textbf{Substep 2}$: Solve
\begin{equation}\label{step2}
	\begin{cases}
		\dfrac{\partial u}{\partial t}=0, \   \mbox{ in } \ \Omega\times(t^n, t^{n+1}),\\
		\dfrac{\partial w}{\partial t}+\gamma(w- \hat{u}^{n+1})=0, \  \mbox{ in } \ \Omega\times(t^n, t^{n+1}),\\
		(u(t^n),w(t^n))=(\hat{u}^{n+1}, \hat{w}^{n+1}),
	\end{cases}
\end{equation}
and set 
$u^{n+1}=u(t^{n+1})$ and $w^{n+1}=w(t^{n+1})$.

Since problem \eqref{step2} is a linear ordinary differential equation, it has the following closed-form solution:
\begin{equation*}
	w^{n+1}=e^{-\gamma\tau}{\hat{w}^{n+1}}+(1-e^{-\gamma\tau})\hat{u}^{n+1}.
\end{equation*}

To solve problem \eqref{step1}, we choose the one-step backward Euler scheme to discretize it, leading to a numerical scheme of the Marchuk–Yanenko type\cite{Splitting}. 

The resulting operator-splitting scheme reads as
\begin{align}
	&\begin{cases}
		\dfrac{u^{n+1}-u^n}{\tau}-\Delta u^{n+1}=f(x,w^n), \ \mbox{ in } \ \Omega,\\
		u=g, \  \mbox{ on } \ \partial\Omega,
	\end{cases} \label{semiO1}\\
	&w^{n+1}=e^{-\gamma\tau}w^n+(1-e^{-\gamma\tau})u^{n+1}.  \label{semiO2}
\end{align}
To initialize the scheme, we compute $u^0$ by solving 
\begin{equation}\label{initial}
	\begin{cases}
		\Delta u^{0}=0,  \ \mbox{ in } \  \Omega,\\
		u^0 = g,  \  \mbox{ on } \  \partial \Omega,
	\end{cases}
\end{equation}
and set $w^0=u^0$.

\subsection{Convergence Analysis}
We analyze convergence of the proposed scheme \eqref{semiO1}-\eqref{semiO2}. Suppose the initial condition $(u^0,w^0)\in H_g^1(\Omega)\times L^2(\Omega)$ is given, and we consider the following weak formulation. 

For $n\geq0$, $u^{n+1} \in H^1_g(\Omega)$ satisfies
\begin{align}
	&(u^{n+1},v)+\tau(\nabla u^{n+1},\nabla v)=(u^n,v)+\tau(f(x,w^n),v), \ \ \forall v \in H_0^1(\Omega), 
	\label{eq.split1.weak}\\
	&w^{n+1}=e^{-\gamma\tau}w^n+(1-e^{-\gamma\tau})u^{n+1}.
	\label{map2}
\end{align}

Suppose $u^* \in H_g^1(\Omega)$ is the weak solution of equation \eqref{s1}, and we  define $w^*=u^*$. For $(u^*,w^*)$, we have 
\begin{align}
	&(u^*,v)+\tau(\nabla u^*,\nabla v)=(u^*,v)+\tau(f(x,w^*),v), \ \ \forall v \in H_0^1(\Omega),
	\label{map101}\\
	&w^*=e^{-\gamma\tau}w^*+(1-e^{-\gamma\tau})u^*. \label{map102}
\end{align}

We make the following assumption on the function $f(x,w)$:
\begin{assumption}\label{assumption}
	There exists a constant $L$ such that for any $w_1,w_2\in L^2(\Omega)$, $f$ satisfies
	\begin{equation}\label{Lipschitz}
		\|f(x,w_1)-f(x,w_2)\|_0\leq L\|w_1-w_2\|_0,
	\end{equation} 
	where the Lipschitz constant $L$ represents the smallest constant such that the inequality \eqref{Lipschitz} holds.
\end{assumption} 

Assumption \ref{assumption} assumes that $f(x,w)$ is Lipschitz with respect to its second argument.  We have the following theorem. 

\begin{theorem}\label{theorem}
	Let $C_1$ be the constant in the Poincar\'{e} inequality \eqref{eq.poincare}, and suppose Assumption \ref{assumption} holds. Let $\{(u^{n},w^{n})\}_{n=0}^{\infty}$ be the sequence of approximate solutions produced by scheme \eqref{eq.split1.weak}-\eqref{map2}. We have the following results:
	
	(i) For $n \geq 0$,
	\begin{equation}\label{con}
		\begin{aligned}
			\|u^{n+1}-u^*\|_0 + \|w^{n+1}-w^*\|_0\leq c^{n+1}\left(\|u^{0}-u^*\|_0 + \|w^{0}-w^*\|_0\right),
		\end{aligned}
	\end{equation} 
	where the constant $c$ is defined to be 
	\begin{equation}       
		c\equiv\max\left\{(2-e^{-\gamma \tau })(1+\dfrac{\tau}{C_1^2})^{-1},\;e^{-\gamma \tau }+(2-e^{-\gamma \tau })(1+\dfrac{\tau}{C_1^2})^{-1}L\tau\right\}.
		\label{eq.thm.proof.c}
	\end{equation} 
	
	(ii) For $n \geq 1$,
	\begin{equation*}
		\begin{aligned}
			|u^{n+1}-u^*|_1\leq c^{n+1/2}(\dfrac{1}{\sqrt{\tau}}+\sqrt{L})(\|u^{0}-u^*\|_0+\|w^{0}-w^*\|_0).\\              
		\end{aligned}
	\end{equation*}
	
	(iii) Assume that $2L<\gamma\leq\dfrac{1}{C_1^2}$ and $\tau\leq\dfrac{1}{\gamma}$. Then the above constant $c<1$ and
	\begin{equation*}
		\begin{aligned}
			\lim_{n\rightarrow \infty}\|u^{n}-u^*\|_1=0, \ \lim_{n\rightarrow \infty}\|w^{n}-w^*\|_0=0. \\            
		\end{aligned}
	\end{equation*}
\end{theorem}
\begin{proof}[Proof of Theorem \ref{theorem}]
	We prove the three statements one by one.
	
	\paragraph{Proof of (i)} From the expressions of \eqref{eq.split1.weak}, \eqref{map2}, \eqref{map101} and \eqref{map102}, we get 
	\begin{align}
		\|w^{n+1}-w^*\|_0&\leq e^{-\gamma \tau}\|w^{n}-w^*\|_0 + (1-e^{-\gamma \tau})\|u^{n+1}-u^*\|_0, \label{1}\\
		(u^{n+1}-u^*,v)&+\tau(\nabla( u^{n+1}-u^*),\nabla v)
		=(u^{n}-u^*,v)\nonumber\\
		&\qquad\qquad\qquad\qquad\qquad\qquad+ \tau(f(x,w^n)-f(x,w^*),v).\label{2} 
	\end{align}	
	Taking $v=u^{n+1}-u^* \in H_0^1(\Omega)$ in \eqref{2}, we obtain
	\begin{equation*}
		\begin{aligned}
			&(u^{n+1}-u^*,u^{n+1}-u^*)+\tau(\nabla( u^{n+1}-u^*),\nabla( u^{n+1}-u^*))\\
			=&(u^{n}-u^*,u^{n+1}-u^*)+ \tau(f(x,w^n)-f(x,w^*),u^{n+1}-u^*). 
		\end{aligned}
	\end{equation*}
	By Assumption \ref{assumption} and the Poincar\'{e} inequality, we deduce that 
	\begin{equation}\label{hahaha}
		\begin{aligned}
			& \|u^{n+1}-u^*\|_0^2+\dfrac{\tau\|u^{n+1}-u^*\|_0^2}{C_1^2} \\
			\leq & \|u^{n+1}-u^*\|_0^2+\tau|u^{n+1}-u^*|_1^2\\
			\leq &\|u^{n}-u^*\|_0\|u^{n+1}-u^*\|_0+ \tau\|f(x,w^n)-f(x,w^*)\|_0\|u^{n+1}-u^*\|_0 \\
			\leq &\|u^{n}-u^*\|_0\|u^{n+1}-u^*\|_0+L \tau \|w^n-w^*\|_0\|u^{n+1}-u^*\|_0,
		\end{aligned}
	\end{equation}
	which can be rewritten as
	\begin{equation}\label{25}
		\begin{aligned}
			\|u^{n+1}-u^*\|_0
			\leq \left(1+\dfrac{\tau}{C_1^2}\right)^{-1}\left(\|u^{n}-u^*\|_0+L \tau \|w^n-w^*\|_0\right).\\
		\end{aligned}
	\end{equation}
	Combining \eqref{1} with  \eqref{25} gives rise to
	\begin{equation*}
		\begin{aligned}
			&\|u^{n+1}-u^*\|_0 + \|w^{n+1}-w^*\|_0\\
			\leq &\, e^{-\gamma \tau }\|w^{n}-w^*\|_0+(2-e^{-\gamma \tau })\|u^{n+1}-u^*\|_0 \\
			\leq&\, (2-e^{-\gamma \tau })(1+\dfrac{\tau}{C_1^2})^{-1}\|u^{n}-u^*\|_0+(e^{-\gamma \tau }+(2-e^{-\gamma \tau })(1+\dfrac{\tau}{C_1^2})^{-1}L\tau)\|w^{n}-w^*\|_0.
		\end{aligned}
	\end{equation*}	
	%	Define the constant $c$ as
	%	\begin{equation}
		%		c=\max\left\{(2-e^{-\gamma \tau })(1+\dfrac{\tau}{C_1^2})^{-1},e^{-\gamma \tau }+(2-e^{-\gamma \tau })(1+\dfrac{\tau}{C_1^2})^{-1}L\tau\right\}.
		%		\label{eq.thm.proof.c}
		%	\end{equation}
	Using definition of the constant $c$ defined in \eqref{eq.thm.proof.c}, we have
	\begin{align*}
		\|u^{n+1}-u^*\|_0 + \|w^{n+1}-w^*\|_0\leq c\,(\|u^{n}-u^*\|_0 + \|w^{n}-u^*\|_0).
	\end{align*}
	Therefore, the above formula suggests that for $n\geq 0$,
	\begin{equation}\label{11}
		\begin{aligned}
			\|u^{n+1}-u^*\|_0 + \|w^{n+1}-w^*\|_0
			\leq c^{n+1}\left(\|u^{0}-u^*\|_0 + \|w^{0}-w^*\|_0\right).
		\end{aligned}
	\end{equation}
	
	\paragraph{Proof of (ii)} Return to formula \eqref{hahaha}. It implies that
	\begin{equation}\label{1212}
		\begin{aligned}
			\tau|u^{n+1}-u^*|_1^2
			\leq \|u^{n}-u^*\|_0\|u^{n+1}-u^*\|_0+L \tau \|w^n-w^*\|_0\|u^{n+1}-u^*\|_0.
		\end{aligned}
	\end{equation}
	Using the result of relation \eqref{11}, we have that, for $n\geq 1$,
	\begin{equation*}
		\begin{aligned}
			&\|u^{n}-u^*\|_0,\;\|w^{n}-w^*\|_0
			\leq c^{n}\left(\|u^{0}-u^*\|_0 + \|w^{0}-w^*\|_0\right),\\
			&\|u^{n+1}-u^*\|_0,\;\|w^{n+1}-w^*\|_0
			\leq c^{n+1}\left(\|u^{0}-u^*\|_0 + \|w^{0}-w^*\|_0\right).
		\end{aligned}
	\end{equation*}
	So inequality \eqref{1212} suggests that 
	\begin{equation*}
		\begin{aligned}
			|u^{n+1}-u^*|_1^2
			\leq c^{2n+1}(\dfrac{1}{\tau}+L)(\|u^{0}-u^*\|_0+\|w^{0}-w^*\|_0)^2.\\
		\end{aligned}
	\end{equation*}
	We deduce that, for $n\geq 1$,
	\begin{equation*}
		\begin{aligned}
			|u^{n+1}-u^*|_1\leq c^{n+1/2}(\dfrac{1}{\sqrt{\tau}}+\sqrt{L})(\|u^{0}-u^*\|_0+\|w^{0}-w^*\|_0).\\            
		\end{aligned}
	\end{equation*}
	
	\paragraph{Proof of (iii)} We derive sufficient conditions to satisfy $c<1$. 
	Consider the first term in (\ref{eq.thm.proof.c}). Define
	\begin{equation*}
		\begin{aligned}
			q(\tau)=\dfrac{\tau}{C_1^2}-1+e^{-\gamma \tau}.
		\end{aligned}
	\end{equation*}
	A sufficient condition for the first term in (\ref{eq.thm.proof.c}) to be smaller than $1$ is 
	\begin{equation*}
		\begin{aligned}
			1-e^{-\gamma\tau}<\frac{\tau}{C_1^2},
		\end{aligned}
	\end{equation*}
	which is equivalent to $q(\tau)>0$. Note that $q(0)=0$ and $q'(\tau)=\dfrac{1}{C_1^2}-\gamma e^{-\gamma \tau}$. Choosing $\gamma\,\leq\,1/C_1^2$ gives rise to $q'(\tau)>0$ for $\tau> 0$, implying that $q(\tau)>0$ for any $\tau>0$.
	
	We next derive a sufficient condition to make the second term in (\ref{eq.thm.proof.c}) be smaller than $1$. To start with, we suppose $\gamma\tau\leq 1$ and $\gamma\leq1/C_1^2$; accordingly, we have
	
	\begin{equation}\label{eq.thm.proof.c2.1}
		\begin{aligned}
			e^{-\gamma \tau }+(2-e^{-\gamma \tau })(1+\dfrac{\tau}{C_1^2})^{-1}L\tau 
			\leq 1-\dfrac{\gamma \tau}{2}+L\tau 
			\leq 1+(L-\dfrac{\gamma}{2})\tau, 
		\end{aligned}
	\end{equation}where the second inequality follows from the fact that $e^{-a}\leq 1-a/2$ for $0\leq a\leq 1$. Thus, a sufficient condition for the right hand side of inequality (\ref{eq.thm.proof.c2.1}) to be smaller than $1$ is	$\gamma>2L$, $\gamma\tau\leq 1$, and $\gamma\leq1/C_1^2$.
	
	Combining the above conditions together, we have $c<1$ if  $2L<\gamma\leq 1/C_1^2$ and $\tau\leq 1/\gamma$. Thus, we can obtain that $\lim\limits_{n\rightarrow \infty}\|u^{n}-u^*\|_1=0$ and $\lim\limits_{n\rightarrow \infty}\|w^{n}-w^*\|_0=0.$
\end{proof}

Theorem \ref{theorem} suggests that the convergence of our iterative algorithm mainly relies on the values of the Lipschitz constant $L$, parameter $\gamma$ and the time step $\tau$, where the Lipschitz constant $L$ depends on equation \eqref{s1}. By definition of the constant $c$ in Theorem \ref{theorem}, we know that, if $\tau$ and $\gamma$ are fixed, then the Lipschitz constant $L$ controls the convergence speed of the iterations, and a larger $L$ may even lead to divergent iterations.
If $2L< 1/C_1^2$ holds, then we can always choose appropriate $\gamma$ and time step $\tau$ to satisfy $2L<\gamma\leq 1/C_1^2$ and $\tau \leq 1/\gamma$. 
In practice, the Poincar\'{e} constant $C_1$ approximately equals the reciprocal of the square root of the smallest eigenvalue of the operator $-\Delta$ in $H_0^1(\Omega)$. In our numerical experiments, the parameter $\gamma$ is selected as the smallest eigenvalue of the operator $-\Delta$ in $H_0^1(\Omega)$ to make $w(t)$ evolve with a speed similar to that of $u(t)$; consequently, the relation  $2L<\gamma\leq 1/C_1^2$ is satisfied naturally. However, if $2L > 1/C_1^2$, there may be no $\gamma$ satisfying those inequalities used in Theorem \ref{theorem} (iii). On the other hand, however, even if this condition is not satisfied, numerical results show that our algorithm is still convergent with a large range of $\gamma$.

\section{Applications of The Proposed Scheme to Fully Nonlinear Elliptic Problems}
\label{sec.MA}
We next apply scheme \eqref{semiO1}-\eqref{semiO2} to solve fully nonlinear elliptic problems that involve eigenvalues of the Hessian. This class of equations include the Monge‐Amp\`ere equation \cite{Caffarelli1}, the Pucci’s equation \cite{Caffarelli2}, and the Minkowski problem \cite{cheng1976regularity}. 

Let $\Db^2u$ denote the Hessian matrix of $u$ in a two-dimensional domain.
The eigenvalues of $\Db^2u$, denoted by  $\lambda_1$ and $\lambda_2$
with $\lambda_1\geq\lambda_2$, can be computed as   
\begin{align}
	\lambda_1=\dfrac{1}{2}\left(\Delta u+\sqrt{|\Delta u|^2-4\det\Db^2u}\right), \quad
	\lambda_2=\dfrac{1}{2}\left(\Delta u-\sqrt{|\Delta u|^2-4\det\Db^2u}\right),
	\label{lambda}
\end{align}
where $\lambda_i=\lambda_i(\Db^2u)$ for $i=1,2$ denote the eigenvalues of the Hessian matrix $\Db^2u$.
The term $|\Delta u|^2-4\det\Db^2u$ is nonnegative since we have 
$(u_{xx}+u_{yy})^2-4(u_{xx}u_{yy}-u_{xy}^2)=(u_{xx}-u_{yy})^2+4u_{xy}^2\geq 0.$

Based on \eqref{lambda}, this class of equations can be reformulated in the form of \eqref{s1}. In this section, we demonstrate how to use this strategy to apply scheme \eqref{semiO1}-\eqref{semiO2} to solve the Monge-Amp\`ere equation and the Pucci's equation. 

\subsection{Monge-Amp\`ere Equation}
The Monge-Amp\`ere equation with Dirichlet boundary conditions is stated as follows:
\begin{equation}\label{ma1}
	\begin{cases}
		\det\Db^2u=f \  \mbox{ in } \ \Omega,\\
		\mbox{$u$ \  is \  convex},\\
		u=g \ \mbox{ on } \ \partial\Omega, 
	\end{cases}
\end{equation}
where $f>0$ and $\Omega$ is a 2-D convex domain. 
Under appropriate conditions, the existence of a unique convex solution of (\ref{ma1}) is guaranteed by the following theorem. 

\begin{theorem}[Existence of Classical Solutions, Theorem 1.1 in \cite{caffarelli1984dirichlet}]\label{ex}
	Suppose \(\Omega\) is a strictly convex domain with $C^\infty$ boundary $\partial \Omega$, \(f \) and $g \in C^\infty (\overline{\Omega})$. Then problem \eqref{ma1} has a unique strictly convex solution $u \in C^\infty( \overline{\Omega})$.
\end{theorem}

In order to apply scheme \eqref{semiO1}-\eqref{semiO2} to solve \eqref{ma1}, we utilize (\ref{lambda}) and the relation $\det\Db^2u=\lambda_1\lambda_2=f.$
Specifically, we have
\begin{align*}
	f=\lambda_1\lambda_2
	&=\dfrac{1}{4}\left(\Delta u+\sqrt{|\Delta u|^2-4\det\Db^2u}\right)\left(\Delta u-\sqrt{|\Delta u|^2-4\det\Db^2u}\right)\\
	&=\dfrac{1}{4}\left((\Delta u)^2-(|\Delta u|^2-4\det\Db^2u)\right),
\end{align*}
yielding $ (\Delta u)^2=|\Delta u|^2-4\det\Db^2 u+4f.$ Hence, we rewrite \eqref{ma1} as
\begin{equation}\label{ma2}
	\begin{cases}
		-\Delta u=-\sqrt{|\Delta u|^2-4\det\Db^2u+4f}\  \mbox{ in } \  \Omega,\\
		u=g\  \mbox{ on } \  \partial \Omega,
	\end{cases}
\end{equation}
which is similar to the form used in \cite{two}.
Although we have recast the Monge-Amp\`ere equation as the boundary value problem \eqref{ma2}, the resulting problem is not of the semilinear type \eqref{s1}, since the right-hand side in \eqref{ma2} still depends on the Hessian matrix $\Db^2 u$. Theorem \ref{theorem} developed for problem \eqref{s1} is not applicable to problem \eqref{ma2} directly. Nonetheless, by treating the Hessian matrix $\Db^2 u$ explicitly in the right-hand side in \eqref{ma2}, we can apply the proposed algorithm designed for the semilinear equation to solve problem \eqref{ma2} numerically.

By Theorem \ref{ex}, we can prove that  boundary value problems \eqref{ma1} and \eqref{ma2} are equivalent.
\begin{corollary}\label{ma-le}
	Suppose \(\Omega\) is a strictly convex domain with $C^\infty$ boundary $\partial \Omega$, \(f \) and $g \in C^\infty (\overline{\Omega})$.
	Problem \eqref{ma2} has a unique solution $u \in C^\infty( \overline{\Omega})$, which is also the convex solution of problem \eqref{ma1}. In other words, problems \eqref{ma1} and \eqref{ma2} have the same unique solution.
\end{corollary}
\begin{proof} [Proof of Corollary \ref{ma-le}]   
	On the one hand, let $u^*$ be the strictly convex solution of \eqref{ma1}. Then $\Db^2u^*$ is positive definite and we have 
	\begin{equation*}
		\sqrt{|\Delta u^*|^2-4\det\Db^2u^*+4f}= \sqrt{|\Delta u^*|^2}=\Delta u^*>0.
	\end{equation*}
	The solution of problem \eqref{ma1} is also the solution of problem \eqref{ma2}.
	
	On the other hand, let $u^*$ be the classical solution of problem \eqref{ma2}. Then we have  
	\begin{equation*}
		\lambda_1(\Db^2u^*)+\lambda_2(\Db^2u^*)=\Delta u^*=\sqrt{|\Delta u^*|^2-4\det\Db^2u^*+4f}\geq 0,
	\end{equation*}
	where $\lambda_i=\lambda_i(\Db^2u^*)$ for $i=1,2$ denote the eigenvalues of the Hessian matrix $\Db^2u^*$. Thus, $u^*$ satisfies $\det\Db^2u^*=f$, i.e.
	\begin{equation*}
		\lambda_1(\Db^2u^*)\lambda_2(\Db^2u^*)=\det\Db^2u^*=f>0.
	\end{equation*}	
	We deduce that $u^*$ is the strictly convex classical solution of problem \eqref{ma1}. 
	
	Thus equations \eqref{ma1} and \eqref{ma2} have the same set of solutions. 
	According to Theorem \ref{ex}, problem \eqref{ma1} has a unique solution. Thus problem \eqref{ma2} has the same unique solution as problem \eqref{ma1}.   
\end{proof}

Corollary \ref{ma-le} demonstrates that the convexity of the solution is implied in the reformulation. Thus it is unnecessary to impose convexity by a projection step as suggested in \cite{Liu2,Liu3}; see the following remark.

\begin{remark}
	In \cite{Liu2}, the numerical approach for the Monge-Amp\`ere equation is based on the following reformulation:
	\begin{equation}\label{cofactor}
		\begin{aligned}	
			-\nabla\cdot({\rm cof}(\mathbf{D}^2u)\nabla u)+2f=0,
		\end{aligned}
	\end{equation}
	where ${\rm cof}(\mathbf{D}^2u)$ is the cofactor matrix of $\Db^2u$. If a convex solution $u$ of the Monge‐Amp\`ere equation  satisfies
	\eqref{cofactor}, the concave solution $-u$
	also  satisfies \eqref{cofactor} when the boundary condition $g=0$ on $\partial \Omega$. As a result, it is necessary for the algorithm in \cite{Liu2} to enforce the convexity of numerical solutions via a projection step.

	For the boundary value problem \eqref{ma1}, if $f$ and $g$ satisfy appropriate conditions, a concave solution exists. We can obtain this concave solution by rewriting problem \eqref{ma1} as
	\begin{equation*}
		\begin{cases}
			-\Delta u=\sqrt{|\Delta u|^2-4\det\Db^2u+4f}\  \mbox{ in } \  \Omega,\\
			u=g\  \mbox{ on } \  \partial \Omega.
		\end{cases}
	\end{equation*}
	With this reformulation, the computational framework developed for computing convex solutions applies to the concave case directly.
	
\end{remark}

Based on the reformulation (\ref{ma2}), we introduce  an auxiliary function $w=u$ and define 
$$F(\Db^2w)=-\sqrt{|\Delta w|^2-4\text{det}\mathbf{D}^2w+4f}.$$ Problem (\ref{ma2}) is equivalent to the following system of PDEs 
\begin{equation}\label{system}
	\begin{cases}
		\begin{cases}
			-\Delta u= F(\Db^2w)
			\  \mbox{ in } \ \Omega,\\
			u=g \   \mbox{ on } \  \partial \Omega, \\
		\end{cases}\\
		w=u, \  \mbox{ in } \ \Omega.\\
	\end{cases}
\end{equation}
Given an initial condition $(u^0, w^0) = (u_0, w_0)$, applying scheme \eqref{semiO1}-\eqref{semiO2} to \eqref{system} leads to the following: for $n > 0$,
\begin{align}\label{maO1}
	&\begin{cases}
		\dfrac{u^{n+1} - u^n}{\tau} - \Delta u^{n+1} = F(\Db^2w^n) \ \mbox{ in } \  \Omega,\\
		u = g \  \mbox{ on } \  \partial \Omega,
	\end{cases}\\
	&w^{n+1} = e^{-\gamma \tau} w^n + (1 - e^{-\gamma \tau}) u^{n+1}.\label{maO2}
\end{align}
To initialize scheme \eqref{maO1}-\eqref{maO2}, we compute $u^0$ by solving 
$\Delta u^{0}=f$ and setting $w^0=u^0$.

The following theorem shows that the solution $u^*$ of \eqref{ma1} is a steady state solution of scheme (\ref{maO1})-(\ref{maO2}):
\begin{theorem}\label{le2}
	If the solution $u^*$ of \eqref{ma1} is in $C^\infty(\overline{\Omega})$ and $w^*=u^*$, then $(u^*,w^*)$ is a steady state solution of scheme (\ref{maO1})-(\ref{maO2}). 
\end{theorem}

\begin{proof}[Proof of Theorem \ref{le2}]
	Since $u^*=w^*$ is the convex solution of \eqref{ma1} with the associated boundary condition $g$, we have
	$|\Delta w^*|^2-4\text{det}\mathbf{D}^2w^*+4f>0$.
	Then $F(\Db^2w^*)=-\sqrt{|\Delta w^*|^2-4\text{det}\mathbf{D}^2w^*+4f}$ is well defined. 
	
	Assuming that $(u^n, w^n)=(u^*, w^*)$, we have $-\Delta u^{n}=F(\Db^2w^n)$. Since scheme  (\ref{maO1}) implies that
	\begin{align*}
		u^{n+1}-\tau\Delta u^{n+1}=u^n+\tau F(\Db^2w^n)=u^n-\tau\Delta u^{n},
	\end{align*}
	we have $(u^{n+1}-u^{n})-\tau \Delta (u^{n+1}-u^n)=0$. Because $u^{n+1}-u^{n} \in H_0^1(\Omega)$ and the operator $I-\tau \Delta$ is coercive in our setting, we deduce that $u^{n+1}=u^n=u^*$.
	Similarly, we use scheme (\ref{maO2}) to obtain
	\begin{align*}
		w^{n+1}=e^{-\gamma\tau}w^n+(1-e^{-\gamma\tau})u^{n+1}=e^{-\gamma\tau}w^*+(1-e^{-\gamma\tau})u^{*}=w^*=w^n.
	\end{align*}
	Therefore, $(u^{n+1}, w^{n+1})=(u^n, w^n)$ and $(u^*,w^*)$ is a steady state solution of scheme (\ref{maO1})-(\ref{maO2}).
\end{proof}

\subsection{Pucci's Equation}
The Pucci's equation, defined by a linear combination of the Pucci's extremal operators, is another fully nonlinear equation. 
%This equation plays a crucial role in various fields, including stochastic control with variable diffusion coefficients \cite{Bensoussan}, and population models \cite{Quitalo, Caffarelli2}. 
%\begin{definition}[Pucci's extremal operators \cite{Caffarelli1}]
%	Letting $0<a<A$, Pucci's extremal operators are defined by
%	\begin{equation*}
	%		\mathcal{M}^{\pm}_{a,A}(M) = A \sum_{\pm \lambda_i > 0} \lambda_i + a \sum_{\pm \lambda_i < 0} \lambda_i, 
	%	\end{equation*}
%	where $M$ is a $N \times N$ symmetric matrix and $\{\lambda_i\}_{i=1}^N$ denote its eigenvalues.
%\end{definition}
In the case $d = 2$, the Pucci's (maximal) equation for $u$ takes the following form equipped with a Dirichlet boundary condition:
\begin{equation}\label{pucci}
	\begin{cases}
		\alpha \lambda_1+\lambda_2=0 \ \mbox{ in }  \ \Omega, \\
		u=g \ \mbox{ on } \  \partial\Omega,
	\end{cases}
\end{equation}
where $\alpha \in (1,+\infty)$, and $\lambda_1$ and $\lambda_2$ with $\lambda_1\geq \lambda_2$ are the two eigenvalues of $\Db^2u$. If $\alpha=1$, the Pucci's equation reduces to a Poisson-Dirichlet problem. If $\alpha>1$, the Pucci's equation implies that
$ \lambda_1\geq0$ and $\lambda_2\leq 0. $

Applying relation \eqref{lambda}, we rewrite the Pucci's equation into a semilinear elliptic form of \eqref{s1} as:
\begin{equation*}
	\begin{cases}
		-\Delta u=\dfrac{\alpha-1}{\alpha+1}\sqrt{|\Delta u|^2-4\det\Db^2u}\  \mbox{ in } \  \Omega,\\
		u=g\  \mbox{ on } \ \partial \Omega.
	\end{cases}
\end{equation*}
We have chosen this formulation based on the following observation. For any $2\times 2$ symmetric Hessian matrix $\Db^2u$ with eigenvalues $\lambda_1$ and $\lambda_2$, we have $|\Delta u|^2-4\det\textbf{D}^2u =(\lambda_1-\lambda_2)^2\geq 0$ 
so that our reformulation is always well defined for any function $u\in C^2(\Omega)$, and the argument of the square root is non-negative without imposing any extra constraint.

We now introduce an auxiliary function $w=u$ and define  $$F(\Db^2w)=\dfrac{\alpha-1}{\alpha+1}\sqrt{|\Delta w|^2-4\det\Db^2w}.$$ 
Then we have the following PDE system:
\begin{equation*}
	\begin{cases}
		\begin{cases}
			-\Delta u= F(\Db^2w)
			\ \mbox{ in } \ \Omega,\\
			u=g\  \mbox{ on } \ \partial \Omega, \\
		\end{cases}\\
		w=u, \  \mbox{ in } \ \Omega.\\
	\end{cases}
\end{equation*}
Given an initial condition $(u^0, w^0) = (u_0, w_0)$, applying scheme \eqref{semiO1}-\eqref{semiO2} to the above system gives rise to the following: for $n > 0$,
\begin{align}\label{pucciO1}
	&\begin{cases}
		\dfrac{u^{n+1} - u^n}{\tau} - \Delta u^{n+1} = F(\Db^2w^n) \  \mbox{ in } \  \Omega,\\
		u = g \ \mbox{ on } \  \partial \Omega,\\
	\end{cases}\\
	&w^{n+1} = e^{-\gamma \tau} w^n + (1 - e^{-\gamma \tau}) u^{n+1}.\label{pucciO2}
\end{align}
In our implementation, we compute the initial condition $u^0$ by formula \eqref{initial} and setting $w^0=u^0$.

Following the proof of Theorem \ref{le2}, we can show that the solution of  (\ref{pucci}) is a steady state solution of scheme (\ref{pucciO1})-(\ref{pucciO2}).

\section{Finite Element Approximation}\label{sec.FEM}
The shared structure of equations \eqref{semiO1}, \eqref{maO1}, and \eqref{pucciO1} motivates us to use the mixed finite element method to approximate both the solution $(u,w)$ and the Hessian matrix $\Db^2w$. 

% in our numerical scheme. Since our algorithms target fully nonlinear equations and need to approximate the Hessian 
% $\Db^2w$, we propose a new numerical method to approximate second derivatives within the linear finite element space.

\subsection{Finite Element Space}
We first introduce the linear finite-element space that we will use. Let $\mathcal{T}_h$ be a quasi-uniform triangulation of the domain $\Omega$ as described in \cite{Brennerbook}, where $h$ denotes the discretization parameter. 
We use $\sum_h=\{Q_k\}_{k=1}^{N_{h}}$ to denote the set of nodes and
$\sum_{0h}=\{Q_k\}_{k=1}^{N_{0h}}$ to denote the set of interior nodes on $\mathcal{T}_h$, respectively. The boundary nodes are represented by $\{Q_k\}_{k=N_{0h}+1}^{N_{h}}$.

Let $V_h$ be the piecewise-linear Lagrange finite-element space defined on $\mathcal{T}_h$. 
We next define some function spaces that we will use:
\begin{equation*}
	\begin{aligned}
		&V_{gh}=\left\{v|v\in V_h,\; v(Q_k)=g(Q_k),\;\forall k=N_{0h}+1,\cdots,N_{h}  \right\}, \\
		&V_{0h}=V_h\cap H_0^1(\Omega).\\
	\end{aligned}
\end{equation*}
Each vertex $Q_j$ has a corresponding basis function $\phi_j$ such that
$
\phi_j(Q_j)=1,\;\; \phi_j(Q_k)=0,\;\;\forall k=1,\;\ldots,N_h,\; k\neq j.
$
The support of $\phi_j$, denoted by $\theta_j$, consists of all triangles having node $Q_j$ as a common node. The area of $\theta_j$ is denoted by $|\theta_j|$.

Here, we introduce several different types of triangulation $\mathcal{T}_h$ of $\Omega$, as illustrated in Figure \ref{mesh}, and these meshes will be used in our numerical experiments. 

\begin{figure}[t!]
	\centering
	\begin{tabular}{c@{\hspace{1.5cm}}c@{\hspace{1.5cm}}c}
		(a)&(b) &(c)\\
		\includegraphics[clip, trim = {80 20 60 10}, width=0.2\textwidth]{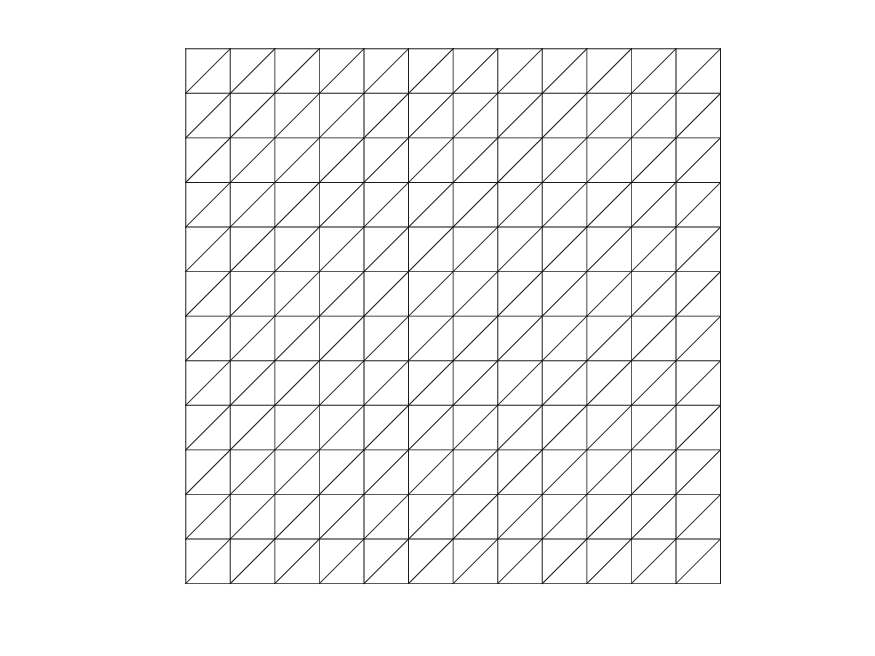}  & 
		\includegraphics[clip, trim = {80 20 60 10}, width=0.2\textwidth]{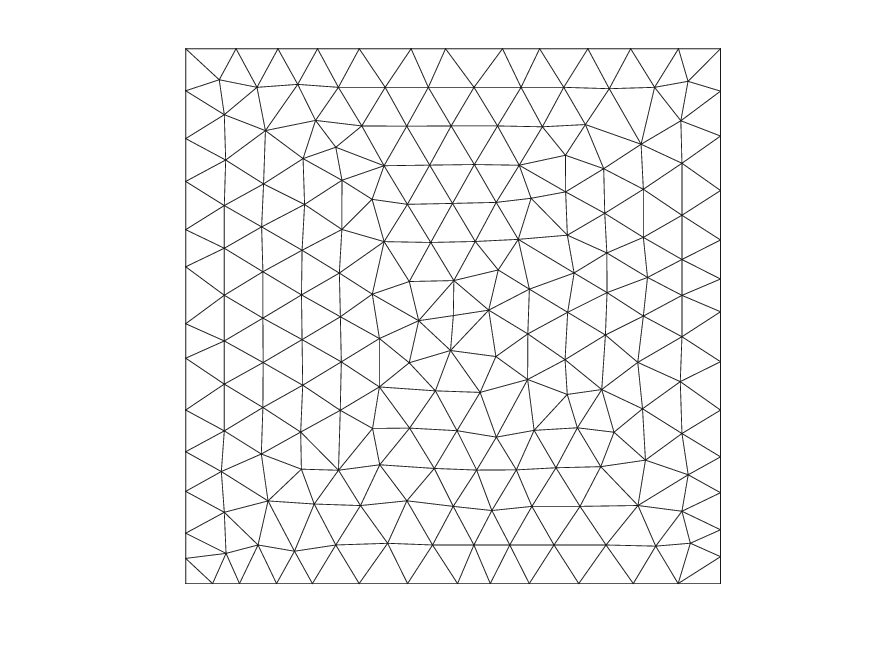} & 
		\includegraphics[clip, trim = {80 20 60 10}, width=0.2\textwidth]{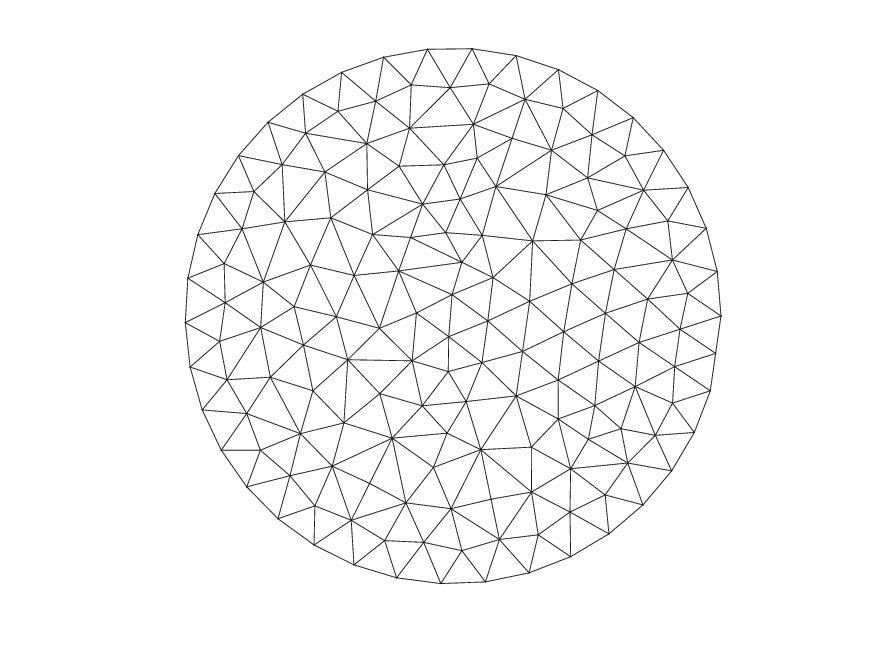}
	\end{tabular}	
	\caption{(a) Regular mesh on the unit square. (b) Unstructured mesh on the unit square. (c) Unstructured mesh on the half-unit disk.}
	\label{mesh}
\end{figure}

\subsection{Approximations of Second-Order Derivatives}
For a function $\psi \in H^2(\Omega)$, the divergence theorem implies that, for $\forall i,j=1,2$
\begin{equation}\label{Green} 
	\int_{\Omega} \frac{\partial^2 \psi}{\partial x_i \partial x_j}\phi dx=-\dfrac{1}{2}\int_{\Omega} \left(\frac{\partial \psi}{\partial x_i}\frac{\partial \phi}{\partial x_j}+\frac{\partial \psi}{\partial x_j}\frac{\partial \phi}{\partial x_i}\right)dx,\
	\forall \phi \in H^1_0(\Omega).
\end{equation}
For $\psi \in V_h$ which is only piecewise linear continuous, we use the relation \eqref{Green} to define a discrete differential operator $D^2_{ijh}$ as an approximation to $\frac{\partial^2}{\partial x_i \partial x_j}$. We do this by finding a function $D^2_{ijh}(\psi) \in V_{0h}$ that satisfies the following condition:
\begin{equation}\label{a1}
	\int_{\Omega} D^2_{ijh}(\psi)\phi dx=-\dfrac{1}{2}\int_{\Omega} \left(\frac{\partial \psi}{\partial x_i}\frac{\partial \phi}{\partial x_j}+\frac{\partial \psi}{\partial x_j}\frac{\partial \phi}{\partial x_i}\right)dx,\  \forall \phi \in V_{0h}.
\end{equation} 

It was mentioned in \cite{Caboussat3} that the approximation formula  \eqref{a1} has two potential weaknesses: 
(i) while the approximation of the Hessian \eqref{a1} has small error at interior nodes, the zero Dirichlet boundary condition results in significant errors at boundary nodes, thereby losing a substantial amount of valuable boundary information; 
(ii) as stated in \cite{Picasso}, the above Hessian recovery technique using linear finite elements has no convergence in general and strongly relies on the types of meshes. The approximation \eqref{a1} only works well on regular meshes; see Figure \ref{mesh}(a).

In our computational experiments, we employ the trapezoidal rule to approximate the integrals. 
Specifically, for a function $\phi(x)$ defined on a regular or unstructured mesh, the trapezoidal rule for integrals means
$ \int_{\Omega}\phi(x) dx\approx\frac{1}{3}\sum_{j=1}^{N_h}|\theta_j|\phi(Q_j)$, 
where $\theta_j$ for $j=1,\cdots, N_h$ are triangles that share the vertex $Q_j$, and $|\theta_j|$ denotes the area of $\theta_j$.
The boundary values of $u$ are set to match 
$g$, and the computation of interior values of $u$ only utilizes interior values of $F(\Db^2w)$ in scheme \eqref{maO1}-\eqref{pucciO1}, independent of the boundary values of the Hessian.
Consequently, the approximation formula \eqref{a1} provides sufficient accuracy for our algorithm when applied on regular meshes. 

However, a simple (Tikhonov) regularization must be performed to ensure convergence in the case of unstructured meshes, as visualized in Figures \ref{mesh}(b) and \ref{mesh}(c). The boundary values of the numerical  Hessian produced by formula \eqref{a1} influence the interior values of the regularized Hessian. Therefore, it is crucial to make a better treatment for boundary nodes of the numerical Hessian before proceeding with regularization.

Since the numerical second-order derivatives \( D^2_{ijh}(\psi) \) computed via (\ref{a1}) have high accuracy at interior vertices, we recompute the boundary values of \( D^2_{ijh}(\psi) \) from their interior values by imposing the homogeneous  Neumann boundary condition. Specifically, the boundary values of \( D^2_{ijh}(\psi) \) are updated by solving
\begin{equation}\label{Neumann}
	\nabla D^2_{ijh}(\psi)\cdot \textbf{n}=0,\ \mbox{on} \  \partial \Omega,
\end{equation}
where $\textbf{n}=(n_1,n_2)$ denotes the unit outward normal vector of $\partial \Omega$.
The following is a detailed procedure for imposing the homogeneous Neumann boundary condition \eqref{Neumann}. 

Suppose values of \( D^2_{ijh}(\psi) \) at interior vertices are computed by solving \eqref{a1}, and these values are denoted by $P_k$ for $k=1,\cdots,N_{0h}$. Accordingly, we would like to determine unknown new values of \( D^2_{ijh}(\psi) \) at boundary vertices, and these values are denoted by $P_k$ for $k=N_{0h}+1,\cdots,N_{h}$. Let $b_1(\psi)$ and $b_2(\psi)$ be the numerical approximation of $\partial_{x_1}D^2_{ijh}(\psi)$ and $\partial_{x_2}D^2_{ijh}(\psi)$ in the linear finite-element space, respectively: Find $b_1(\psi)$ and $b_2(\psi)$ $\in$ $V_h$, satisfying

\begin{align}
	&\int_{\Omega} b_1(\psi)\phi dx=\int_{\Omega} \partial_{x_1}D^2_{ijh}(\psi)\phi dx,\quad \forall \phi \in V_{h}, \label{n1}\\
	&\int_{\Omega} b_2(\psi)\phi dx=\int_{\Omega} \partial_{x_2}D^2_{ijh}(\psi)\phi dx,\quad \forall \phi \in V_{h}.\label{n2}
\end{align}
Since only boundary values of $b_1(\psi)$ and $b_2(\psi)$ are needed according to (\ref{Neumann}), we set the test functions $\phi$ to be the basis functions $\phi_k$, where $k=N_{0h}+1, \cdots, N_h$. 

Next, we apply the trapezoidal rule to approximate the integrals in \eqref{n1}. By taking test functions $\phi_k$ associated with vertex $Q_k$ on $\partial \Omega$, the value of $b_1(\psi)(Q_k)$ is given by
\begin{align*}\label{b1}
	&b_1(\psi)(Q_k) \nonumber \\
	=&\dfrac{3}{|\theta_k|}\int_{\Omega}\partial_{x_1}D^2_{ijh}(\psi)\phi_kdx \nonumber \\
	=&\dfrac{3}{|\theta_k|}\int_{\Omega}\partial_{x_1}\left(\sum_{j=1,\cdots,N_{0h}}P_j\phi_j+\sum_{j=N_{0h}+1,\cdots,N_{h}}P_j\phi_j\right)\phi_kdx \nonumber \\
	=&\dfrac{3}{|\theta_k|}\int_{\Omega}\left(\sum_{j=1,\cdots,N_{0h}}\partial_{x_1}(P_j\phi_j)\phi_k+\sum_{j=N_{0h}+1,\cdots,N_{h}}\partial_{x_1}(P_j\phi_j)\phi_k\right)dx \nonumber \\
	=&\dfrac{3}{|\theta_k|}\int_{\Omega}\left(\sum_{j=1,\cdots,N_{0h}}P_j\partial_{x_1}(\phi_j)\phi_k+\sum_{j=N_{0h}+1,\cdots,N_{h}}P_j\partial_{x_1}(\phi_j)\phi_k\right)dx \nonumber \\  =&\dfrac{3}{|\theta_k|}\sum_{j=1,\cdots,N_{0h}}P_j\int_{\Omega}\partial_{x_1}(\phi_j)\phi_kdx+\dfrac{3}{|\theta_k|}\sum_{j=N_{0h}+1,\cdots,N_{h}}P_j\int_{\Omega}\partial_{x_1}(\phi_j)\phi_kdx.\nonumber  \\ 
\end{align*}
This expression indicates that $b_1(\psi)(Q_k)$ is a linear combination of $P_{N_{0h}+1},\cdots,P_{N_h}$ plus some known constants; in fact, the linear combination only includes $P_k$ and $P_j$ defined on two adjacent boundary nodes. Similarly, the same is true for $b_2(\psi)(Q_k)$. 

Thus relation \eqref{Neumann} can be approximated as: For $\forall k=N_{0h}+1, \cdots, N_h$,
\begin{equation}\label{nn}
	\nabla D^2_{ijh}(\psi)(Q_k)\cdot \textbf{n}(Q_k)=b_1(\psi)(Q_k)n_1(Q_k)+b_2(\psi)(Q_k)n_2(Q_k)=0, 
\end{equation}
where $\textbf{n}(Q_k)$ are computed in advance.
Relation \eqref{nn} leads to a linear system for $P_k$, where $k=N_{0h}+1, \cdots, N_h$.
Therefore, we can get $P_k$ for $k=N_{0h}+1, \cdots, N_h$ by solving the above linear system.
%The new $D^2_{ijh}(\psi)$ is confirmed that $D^2_{ijh}(\psi)(Q_k)=P_k$ for $k=1, \cdots, N_{0h}$ and $D^2_{ijh}(\psi)(Q_k)=P_k$ for $k=N_{0h}+1, \cdots, N_h$.

It is known that the above numerical Hessian has deteriorated accuracy when $h\rightarrow 0$ and it may completely lose accuracy on unstructured meshes \cite{Caboussat3}. Therefore, we use the Tikhonov regularization \cite{Tikhonov} to overcome this difficulty by adding some viscosity as follows:
\begin{equation*}
	\begin{cases}
		-\epsilon\nabla^2\tilde{D}^2_{ijh}(\psi)+\tilde{D}^2_{ijh}(\psi)=D^2_{ijh}(\psi), \  \mbox{ in } \ \Omega,\\
		\dfrac{\partial \tilde{D}^2_{ijh}(\psi) }{\partial \textbf{n}}=0,\ \mbox{ in }\ \partial \Omega.\\
	\end{cases}
\end{equation*}
Its variational form reads as: find $\tilde{D}^2_{ijh}(\psi)\in V_{h}$ for 
$i,j=1,2$ satisfying
\begin{equation}\label{a3}
	\begin{aligned}
		\epsilon\int_\Omega \nabla \tilde{D}^2_{ijh}(\psi)\cdot\nabla \phi dx+\int_{\Omega}\tilde{D}^2_{ijh}(\psi) \phi dx
		=\int_{\Omega}D^2_{ijh}(\psi)\phi dx, \; \forall \phi\in V_{h},\\
	\end{aligned}
\end{equation}
where $\epsilon$ is of order $O(h^2)$ on unstructured meshes.

In summary, the numerical method for second-order derivatives on the unstructured meshes involves three fundamental steps: (i) utilize the interior value based on the divergence theorem as described in equation \eqref{a1}, (ii) impose the Neumann boundary condition using equation \eqref{Neumann}, and (iii) implement the Tikhonov regularization as specified in equation \eqref{a3}. 

\begin{remark}
	It is important to note that the numerical method for computing second-order derivatives on a regular mesh only uses formula \eqref{a1} without imposing the vanishing Neumann boundary condition and implementing Tikhonov regularization, i.e. regularization parameter $\epsilon=0$. In fact, on a regular mesh, our algorithm does not rely on the values of the Hessian on the boundary in the sense that boundary values of $D^2w$ are not needed in schemes \eqref{maO1} and \eqref{pucciO1} when integrals are approximated by trapezoidal rule; this subtle point also explains why our current algorithm yields higher accuracy than the algorithm proposed in \cite{Liu2}. Nevertheless, it is only on an unstructured mesh that we are imposing the vanishing Neumann boundary condition and implementing Tikhonov regularization to compute second-order derivatives at boundary vertices.   
\end{remark}

\section{Finite Element Implementation of Numerical Schemes}
\label{sec.implementation}
Now we are ready to give fully discrete schemes for the semilinear equation, the Monge-Amp\`ere equation, and the Pucci's equation, where all integrations are computed by the trapezoidal rule. Recall that $V_h$ is the piecewise continuous linear Lagrange finite element space, 
$V_{gh}=\left\{v|v\in V_h, v(Q_k)=g(Q_k),\forall k=N_{0h}+1,\cdots,N_{h}  \right\}$, and
$ V_{0h}=V_h\cap H_0^1(\Omega).$ 

To initialize our scheme in numerical implementation, we use the following discrete analogue of the continuous initialization in terms of piecewise continuous linear Lagrange finite elements, which is summarized in the following:
$$u^{0} \in V_{gh},	\ \int_{\Omega}\nabla u^{0}\cdot\nabla vdx=
-\int_{\Omega}fvdx, \ \forall v\in V_{0h},$$
and set $w^0=u^0$. For the semilinear equation and Pucci's equation, we set $f=0$ in the above.

\subsection{Implementing Scheme \eqref{semiO1}-\eqref{semiO2}}
Given an initial condition $(u^0,w^0)$, for $n\geq 0$, the scheme \eqref{semiO1}-\eqref{semiO2} for the semilinear elliptic equation is discretized as follows:

\noindent $\textbf{Substep 1}$:
For any $v\in V_{0h}$, find $u^{n+1} \in V_{gh}$ satisfying
\begin{equation}\label{S1}
	\begin{aligned}
		\int_{\Omega} u^{n+1}vdx+\tau \int_{\Omega}\nabla u^{n+1}\cdot\nabla vdx=\int_{\Omega} u^{n}vdx+\tau
		\int_{\Omega}f(x,w^n)vdx.
	\end{aligned}
\end{equation}
$\textbf{Substep 2}$: Compute $w^{n+1} \in V_h$ by
\begin{equation}\label{S2}
	w^{n+1}(Q_k)=e^{-\gamma\tau}w^n(Q_k)+(1-e^{-\gamma\tau})u^{n+1}(Q_k), \ \forall k=1,\cdots,N_{h}.
\end{equation} 

\subsection{Implementing Scheme \eqref{maO1}-\eqref{maO2}}
Given an initial condition $(u^0,w^0)$, for $n\geq 0$, the scheme \eqref{maO1}-\eqref{maO2} for the Monge-Amp\`{e}re equation is discretized as follows:

\noindent $\textbf{Substep 1}$:	For any $v\in V_{0h}$, find $u^{n+1} \in V_{gh}$ satisfying 
\begin{equation}\label{M1}
	\begin{aligned}
		&\int_{\Omega} u^{n+1}vdx+\tau \int_{\Omega}\nabla u^{n+1}\cdot\nabla vdx\\
		=&\int_{\Omega} u^{n}vdx-\tau		\int_{\Omega}\sqrt{(D_{11h}^2w^n+D_{22h}^2w^n)^2-4\det \Db_h^2w^n+4f}vdx. 
	\end{aligned}
\end{equation} 
$\textbf{Substep 2}$: Compute $w^{n+1} \in V_h$ by
\begin{equation}\label{M2} 
	w^{n+1}(Q_k)=e^{-\gamma\tau}w^n(Q_k)+(1-e^{-\gamma\tau})
	u^{n+1}(Q_k),\  \forall k=1,\cdots,N_{h}.
\end{equation} 
Here, the second-order derivatives $D_{11h}^2w^n$ and $D_{22h}^2w^n$ are computed by \eqref{a1}, \eqref{Neumann} and \eqref{a3} in Section 4, and 
$\Db_h^2u=
\begin{pmatrix}
	D^2_{11h}(u) & D^2_{12h}(u) \\
	D^2_{21h}(u) & D^2_{22h}(u)
\end{pmatrix}.
$

\subsection{Implementing Scheme \eqref{pucciO1}-\eqref{pucciO2}}
Given an initial condition $(u^0,w^0)$, for $n\geq 0$, the scheme \eqref{pucciO1}-\eqref{pucciO2} for the Pucci's equation is discretized as follows:

\noindent $\textbf{Substep 1}$:	For any $v\in V_{0h}$, find $u^{n+1} \in V_{gh}$ satisfying 
\begin{equation}\label{P1}
	\begin{aligned}
		&\int_{\Omega} u^{n+1}vdx+\tau \int_{\Omega}\nabla u^{n+1}\cdot\nabla vdx\\
		=&\int_{\Omega} u^{n}vdx+\tau\dfrac{\alpha-1}{\alpha+1}\int_{\Omega}\sqrt{(D_{11h}^2w^n+D_{22h}^2w^n)^2-4\det \Db_h^2w^n}vdx. 
	\end{aligned}
\end{equation}
$\textbf{Substep 2}$: Compute $w^{n+1} \in V_h$ by
\begin{equation}\label{P2}
	w^{n+1}(Q_k)=e^{-\gamma\tau}w^n(Q_k)+(1-e^{-\gamma\tau})
	u^{n+1}(Q_k),\  \forall k=1,\cdots,N_{h}.
\end{equation} 

\section{Numerical Experiments}
\label{sec.experiments}
We conduct a variety of numerical experiments to demonstrate the performance of our proposed algorithms. 
We set the parameters as follows: $\tau=1$ and $\epsilon=0$ for the regular mesh as shown in Figure \ref{mesh}(a), and $\tau=1$ and $\epsilon=h^2$ for the unstructured meshes as shown in Figure \ref{mesh}(b) and (c). The stopping criterion for the proposed algorithm is set as $||u^{n+1}-u^n||_0<10^{-9}$ unless otherwise specified.

We define $\hat{\lambda}_1$ to be the smallest eigenvalue of the stiffness matrix associated with the operator $-\Delta$ on $V_{0h}$. In practice, we set $\gamma = \hat{\lambda}_1/h^2$ to make $w(t)$ evolve with a speed similar to that of $u(t)$. Extensive numerical experiments indicate that numerical convergence is achieved for a broader range of parameter choices. 

We test our proposed algorithms on the semilinear elliptic equation, the Monge‐Amp\`ere equation, and the Pucci’s equation. For all the following numerical examples, $\tau=1$ is used, which yields good accuracy and small number of iterations needed for convergence of the algorithm. We therefore recommend $\tau=1$ as the default choice. If $\tau=1$ fails to converge for a given problem, we suggest reducing the step size accordingly.

\subsection{Semilinear Elliptic Equation} 
We apply scheme \eqref{semiO1}-\eqref{semiO2} to solve the following semilinear elliptic problem,  
\begin{equation}\label{Semilinear}
	\begin{cases}
		- \Delta u= L|u| - \Delta g- L|g|\  \mbox{in} \ \Omega, \\ 
		u = g(x) \ \mbox{on} \ \partial \Omega,
	\end{cases}
\end{equation}
where $g(x)= \cos(\pi x_1) \cos(\pi x_2)$, $\Omega$ is the unit square $(0, 1)^2$, and $L$ is the Lipschitz constant in Assumption \ref{assumption}.
The exact solution is
$
u = \cos(\pi x_1) \cos(\pi x_2) \ \mbox{in} \ \Omega.
$

We report in Table \ref{test-Semilinear} the numerical results on the regular mesh, Figure \ref{mesh}(a), and the unstructured mesh in the unit square, Figure \ref{mesh}(b), when $L=1/2$. As we are using linear finite elements, our proposed algorithm for the semilinear problem preserves the optimal convergence of order $2$ in terms of the $L^2$ error and nearly optimal rate in terms of the $L^\infty$ error on both meshes; in fact, the algorithm is optimal in terms of  $L^\infty$ error on the regular mesh as well.  Since the algorithms for our numerical experiments are implemented with piecewise-linear finite elements, the optimal convergence rate should be second-order in terms of mesh size $h$, which is validated by this example.

Our algorithm converges with only 7 iterations. 
According to the definition of the constant $c$ in Theorem \ref{theorem}, the speed of convergence slows down as $L$ increases. We test the proposed algorithm for problem (\ref{Semilinear}) with different $L$'s on the unstructured mesh,  Figure \ref{mesh}(b). The convergence histories are presented in Figure \ref{test-Semilinear-history}. We observe that our algorithm converges slower as $L$ becomes larger, which agrees with our theory.

\begin{table}[t!]
	\centering
	(a)
	\begin{tabular}{cccccc}
		\hline
		$h$&Iterations& $L^2$ error & Rate     &  $L^{\infty}$ error & Rate\\
		\hline		
		1/10&7 &7.45$\times10^{-4}$ &   &  1.47$\times10^{-3}$ &\\
		
		1/20&7 &1.90$\times10^{-4}$ & 1.97 & 3.72$\times10^{-4}$ & 1.98\\
		1/40&7 &4.76$\times10^{-5}$ & 2.00 &  9.34$\times10^{-5}$& 1.99\\
		1/80&7 &1.19$\times10^{-5}$ & 2.00  & 2.34$\times10^{-5}$& 2.00\\
		\hline
	\end{tabular} \\ \vspace{0.1cm}
	(b)
	\begin{tabular}{cccccc}
		\hline
		$h$&Iterations& $L^2$ error & Rate     &  $L^{\infty}$ error & Rate\\
		\hline		
		1/10&7 &1.94$\times10^{-3}$ &   &  6.22$\times10^{-3}$ &\\
		
		1/20&7 &3.93$\times10^{-4}$ & 2.30  & 1.68$\times10^{-3}$ & 1.89\\
		1/40&7 &1.22$\times10^{-4}$ &  1.69 &  4.97$\times10^{-4}$&1.76\\
		1/80&7 &2.91$\times10^{-5}$ & 2.07   & 1.52$\times10^{-4}$&1.71\\
		\hline
	\end{tabular} \\ \vspace{0.1cm}
	\caption{(Semilinear equation.) Numerical results for problem \eqref{Semilinear} with $L=1/2$ on the unit square $(0, 1)^2$. (a) The regular mesh. (b) The unstructured mesh of the unit square.}
	\label{test-Semilinear}
\end{table}

\begin{figure}[t!]
	\centering
	\begin{tabular}{c@{\hspace{1.5cm}}c}    
		(a) & (b) \\
		\includegraphics[width=0.35\textwidth]{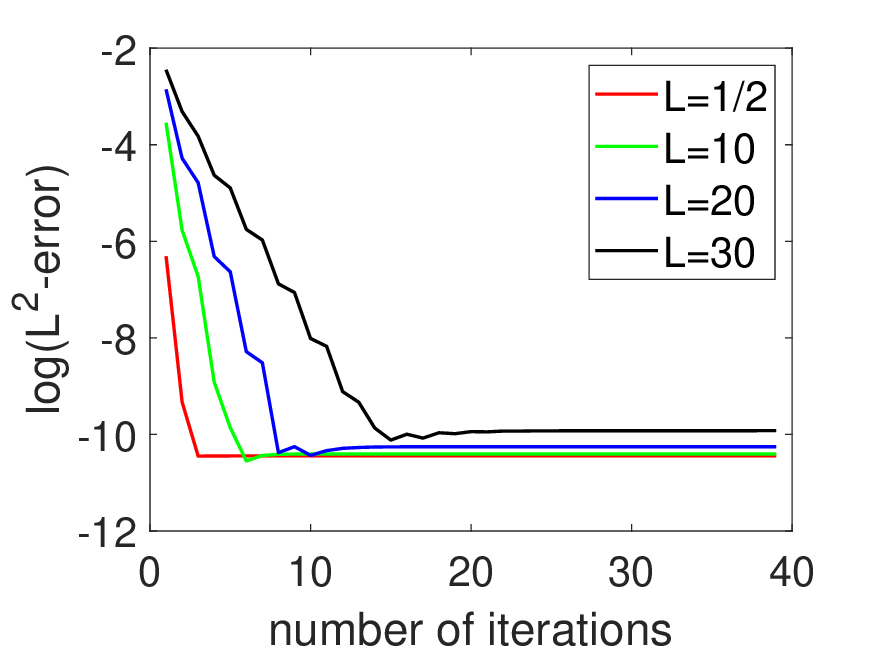} &  
		\includegraphics[width=0.35\textwidth]{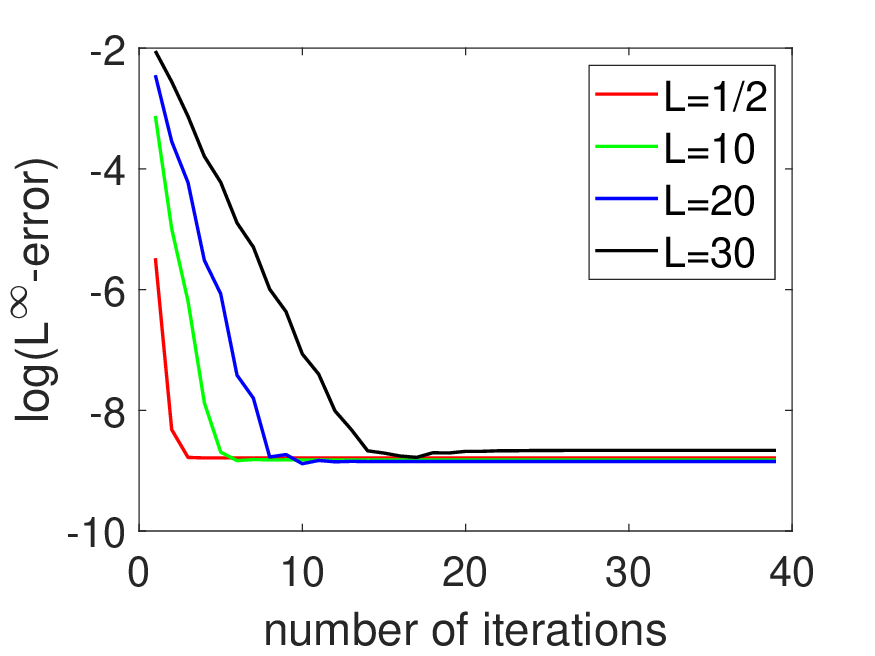} \\
	\end{tabular}    
	\caption{(Semilinear equation.) $h=1/80$. Histories of (a) $L^2$ errors and (b) $L^\infty$ errors for problem \eqref{Semilinear} on the unstructured mesh of the unit square $(0, 1)^2$. The equation with $L=\frac{1}{2}, 10, 20$, and $30$, respectively, is solved. }
	\label{test-Semilinear-history}
\end{figure}

\begin{figure}[t!]
	\centering
	\begin{tabular}{c}    		\includegraphics[width=0.4\textwidth]{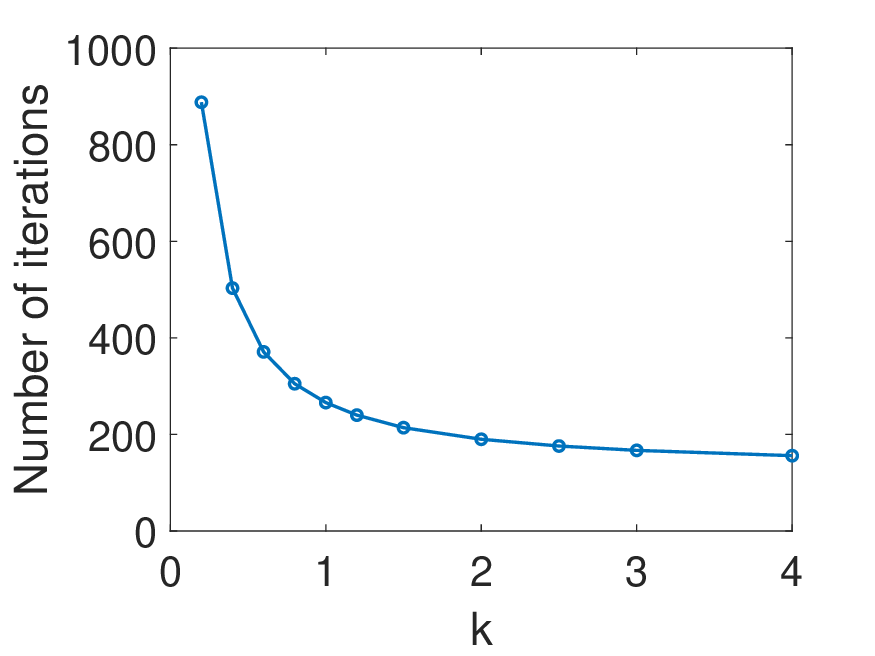}  \\
	\end{tabular}    
	\caption{(Semilinear equation.) Numerical results for problem \eqref{Semilinear} with $L=30$ and $\tau=h$ on the unstructured mesh of the unit square $(0, 1)^2$: the parameter $\gamma = k \hat{\lambda}_1/h^2$, where $k$ is a scaling factor.}
	\label{kt}
\end{figure}
\begin{table}[t!]
	\centering
	\begin{tabular}{ccccccccc}
		\hline
		$h$& Iter($\tau=h$) &  Iter($\tau=10h$) & Iter($\tau=1$) & Iter($\tau=10$) &  $L^{2}$ error\\
		\hline
		1/20& 307 &129 &161 &NaN & 1.39$\times10^{-3}$
		\\
		1/40&438 &129 &121 &138 & 6.26$\times10^{-4}$
		\\
		1/80& 726 &147 &113 &127 & 1.37$\times10^{-4}$
		\\
		\hline
	\end{tabular} \\ \vspace{0.1cm}
	\caption{(Semilinear equation.) Numerical results with different time steps $\tau=h, 10h, 1, 10$ for problem \eqref{Semilinear} with $L=42$ on the unstructured mesh of the unit square $(0, 1)^2$.}
	\label{iteration-tau}
\end{table}

Furthermore, we study how the choice of $\gamma$ will affect the convergence of the proposed algorithm, where we take $L=30$ and $\tau=h$ as a particular example.
We set $\gamma = k \hat{\lambda}_1/h^2$, where $k$ is a scaling factor. The number of iterations required for convergence with different $k$ is reported in Figure \ref{kt}.
It indicates that when $k<1$, the number of iterations required is very sensitive to $\gamma$. The reason is that with small $\gamma$, the evolution speed of $w$ is slower than that of $u$ when $k<1$, and the overall number of iterations is determined by the evolution speed of $w$. On the other hand, Figure \ref{kt} indicates that when $k>1$, the value of $\gamma$ has minor impact on the number of iterations. In this case, the evolution speed of $w$ is faster than that of $u$, and the number of iterations is determined by the evolution speed of $u$. Therefore, $\gamma =\hat{\lambda}_1/h^2$ is a simple and effective choice. 

Lastly, to assess the sensitivity of the proposed algorithm to the time step $\tau$, in Table \ref{iteration-tau} we show the numbers of iterations and $L^2$ errors for $\tau=h, 10h, 1, 10$, respectively, for problem \eqref{Semilinear} with $L = 42$ on an unstructured mesh. The results indicate that increasing $\tau$ reduces the number of iterations required for convergence, and within this particular range of $\tau$, the accuracy of the computed solution is essentially not sensitive to $\tau$. However, once $\tau$ exceeds a certain threshold, the method may lose stability and fail to converge.

\subsection{Monge-Amp\`ere Equation}
We apply scheme \eqref{maO1}-\eqref{maO2} to solve the Monge‐Amp\`{e}re equation, and we compare the new scheme with the direct operator splitting (DOS) method based on the divergence form \cite{Liu2}, and the nonlinear Gauss‐Seidel iteration based 
finite-difference (FD) method \cite{two}.  In this section, the stopping criterion for the DOS algorithm is also set as $||u^{n+1}-u^n||_0<10^{-9}$; the FD method needs a smaller stopping criterion to achieve convergence, where the Gauss-Seidel iteration stops when $||u^{n+1}-u^n||_0$ is less than $10^{-14}$.

\subsubsection{A Quadratic Solution}
The first example for the Monge-Amp\`ere equation is defined by
\begin{equation}\label{Polynomial}
	\begin{cases}
		\det \Db^2u=256\ \mbox{in} \ \Omega,\\
		g=8\left(\beta(x_1-\frac{1}{2})^2+\dfrac{1}{\beta}(x_2-\frac{1}{2})^2\right)-1\ \mbox{on}\ \partial \Omega,
	\end{cases}
\end{equation}
where $\Omega=(0,1)^2$, a unit square. The exact solution $u$ is a quadratic function given by
$
u=8(\beta(x_1-\frac{1}{2})^2+\dfrac{1}{\beta}(x_2-\frac{1}{2})^2)-1 \ \mbox{in} \  \Omega.
$

\begin{table}[t!]
	\centering
	\begin{tabular}{c|cc|cc}
		\hline
		& \multicolumn{2}{c|}{$\beta=1$} & \multicolumn{2}{c}{$\beta=4$} \\
		\hline
		$h$& $L^2$ error &  $L^{\infty}$ error& $L^2$ error &  $L^{\infty}$ error\\
		\hline
		1/10&2.35$\times10^{-16}$ & 6.66$\times10^{-16}$&1.95$\times10^{-14}$ & 5.51$\times10^{-14}$ \\
		1/20&3.87$\times10^{-15}$ & 1.08$\times10^{-14}$&1.95$\times10^{-14}$ & 5.51$\times10^{-14}$
		\\
		1/40&6.05$\times10^{-15}$ & 1.63$\times10^{-14}$&2.58$\times10^{-14}$ & 7.28$\times10^{-14}$
		\\
		1/80&4.56$\times10^{-14}$ & 1.18$\times10^{-13}$&2.91$\times10^{-13}$ & 7.29$\times10^{-13}$
		\\
		\hline
	\end{tabular} \\ \vspace{0.1cm}
	\caption{(Monge‐Amp\`{e}re equation.) Numerical results for problem \eqref{Polynomial} with $\beta=1$ and $\beta=4$, respectively, on the regular mesh as shown in  Figure \ref{mesh}(a).}
	\label{Polynomial_error}
\end{table}
\begin{figure}[t!]
	\centering
	\begin{tabular}{c@{\hspace{1.5cm}}c}
		(a) & (b)\\
		\includegraphics[width=0.35\textwidth]{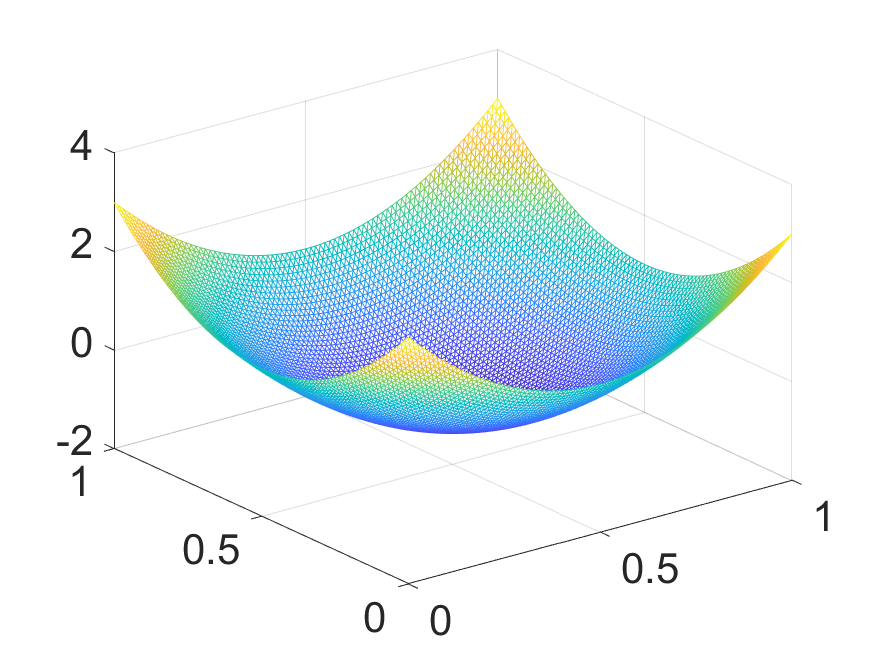} &
		\includegraphics[width=0.35\textwidth]{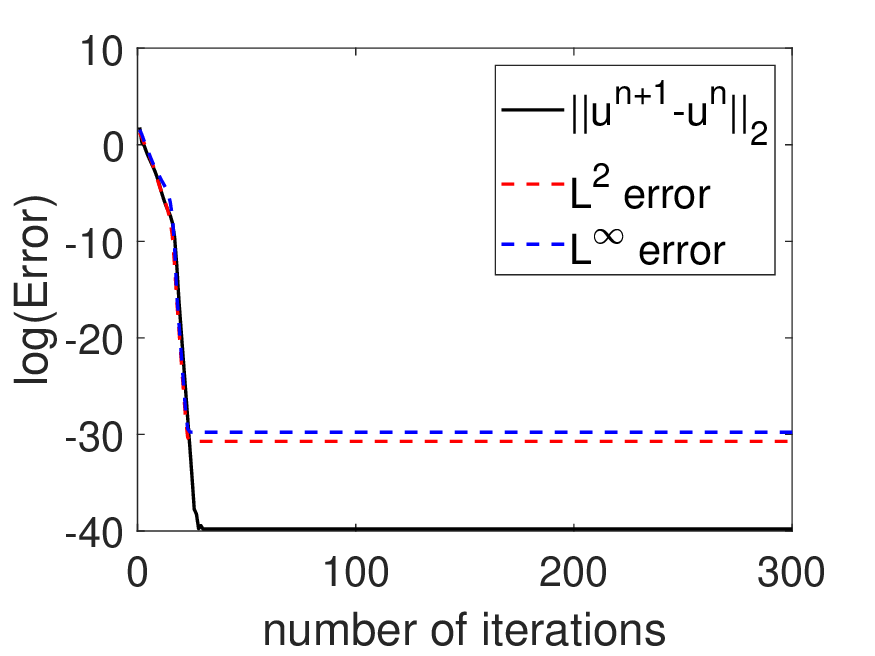}\\
		(c) & (d)\\
		\includegraphics[width=0.35\textwidth]{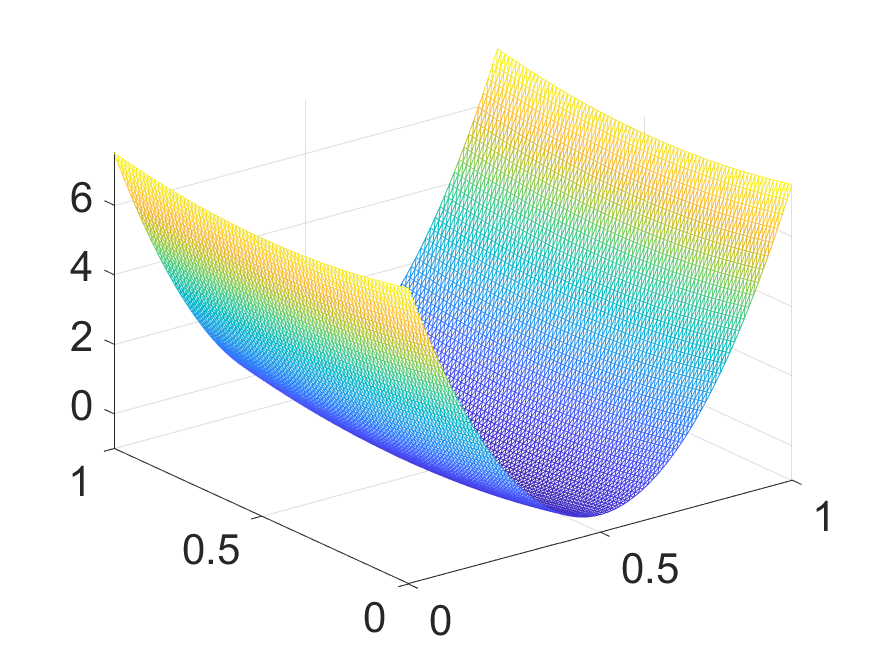} &
		\includegraphics[width=0.35\textwidth]{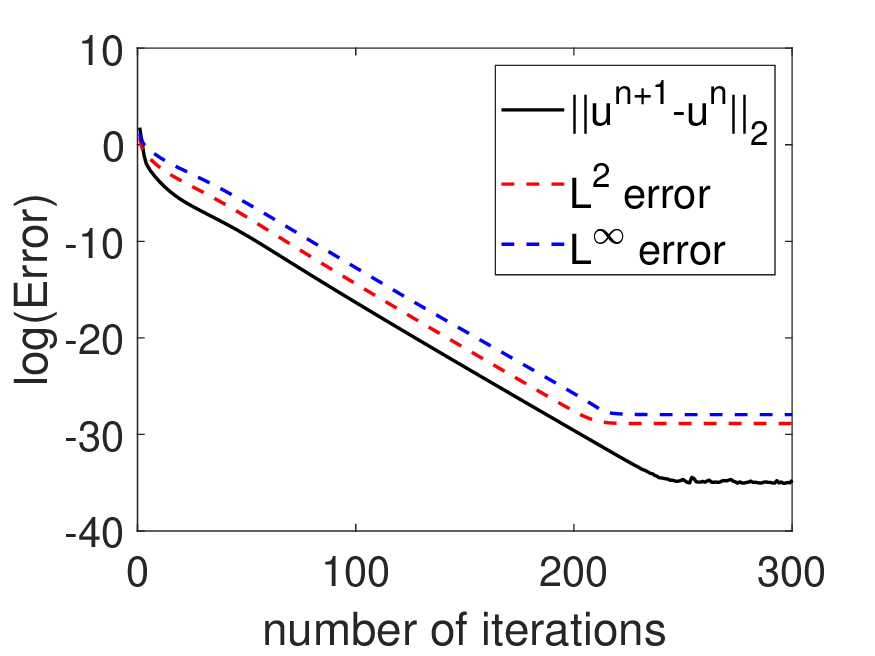}\\
	\end{tabular}
	\caption{(Monge‐Amp\`{e}re equation.) Numerical results for problem \eqref{Polynomial} on the regular mesh with $h=1/80$: (a) Graph of the computed solution for $\beta=1$. (b) Error history for $\beta=1$. (c) Graph of the computed solution for $\beta=4$. (d) Error history for $\beta=4$. }
	\label{Polynomial_graph}
\end{figure}
We apply our proposed algorithm to the problem on the regular mesh as shown in Figure \ref{mesh}(a). The computational results are presented in Table \ref{Polynomial_error} and Figure \ref{Polynomial_graph}. 

Table \ref{Polynomial_error} shows the approximation errors, where Columns  2-3 present the $L^2$ and $L^{\infty}$ errors for $\beta=1$, and Columns 4-5 present the results for $\beta=4$. In this experiment, our scheme attains machine-precision accuracy for both $\beta=1$ and $\beta=4$. This exceptional performance can be attributed to the nature of the exact solution $u$, which is a quadratic function with constant second-order derivatives. As a consequence, our numerical approximation scheme for the Hessian in Section \ref{sec.FEM} captures these second-order derivatives with perfect accuracy in interior nodes, and the computed solution satisfies the Monge-Amp\`ere equation exactly (up to machine precision). 

Figure \ref{Polynomial_graph} illustrates the approximation results and error histories for various values of $\beta$, where the mesh parameter $h=1/80$. The numerical solution successfully captures the convexity on the regular mesh. As the parameter $\beta$ in problem \eqref{Polynomial} increases, the solution exhibits stronger anisotropic characteristics. This enhanced anisotropy leads to a substantial increase in the required number of iterations for convergence.

\subsubsection{A Smooth Example}
We consider the Monge-Amp\`ere equation defined as
\begin{equation}\label{expfunction}
	\begin{cases}
		\det \textbf{D}^2u=(1+|x|^2)e^{|x|^2}\ \mbox{in} \ \Omega,\\
		g=e^{\frac{|x|^2}{2}}\ \mbox{on}\ \partial \Omega,
	\end{cases}
\end{equation}
where $\Omega=(0,1)^2$, a unit square. The solution $u$ is given as
$
u=e^{\frac{|x|^2}{2}} \ \mbox{in} \  \Omega.
$
We first test DOS, FD and our proposed algorithm on the regular mesh as shown in Figure \ref{mesh}(a). In the experiment, we set $\tau=2h^2$ in the DOS algorithm so that it converges. 

\begin{figure}[t!]
	\centering
	\begin{tabular}{c@{\hspace{1.5cm}}c}
		(a) & (b)\\
		\includegraphics[width=0.35\textwidth]{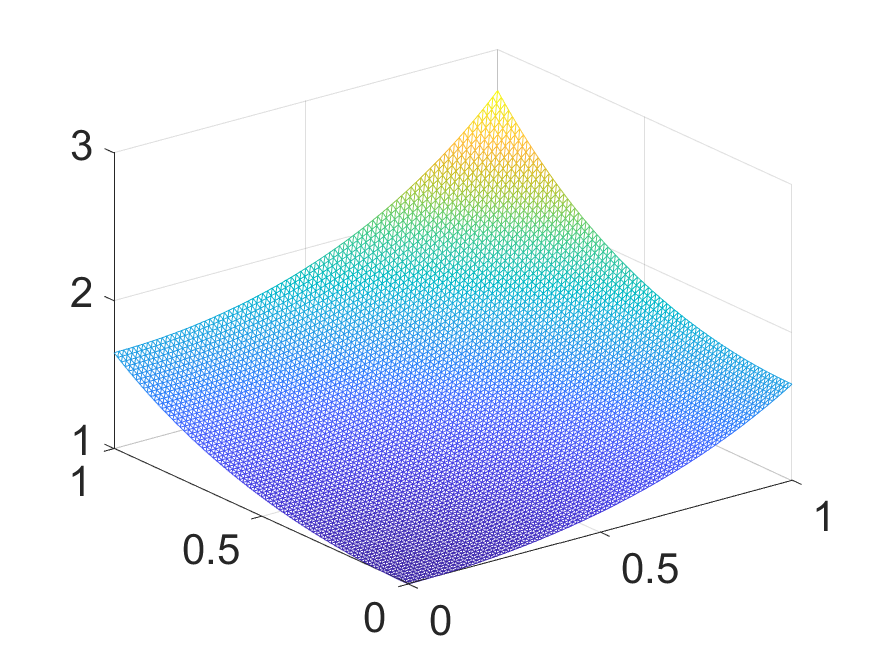} &
		\includegraphics[width=0.35\textwidth]{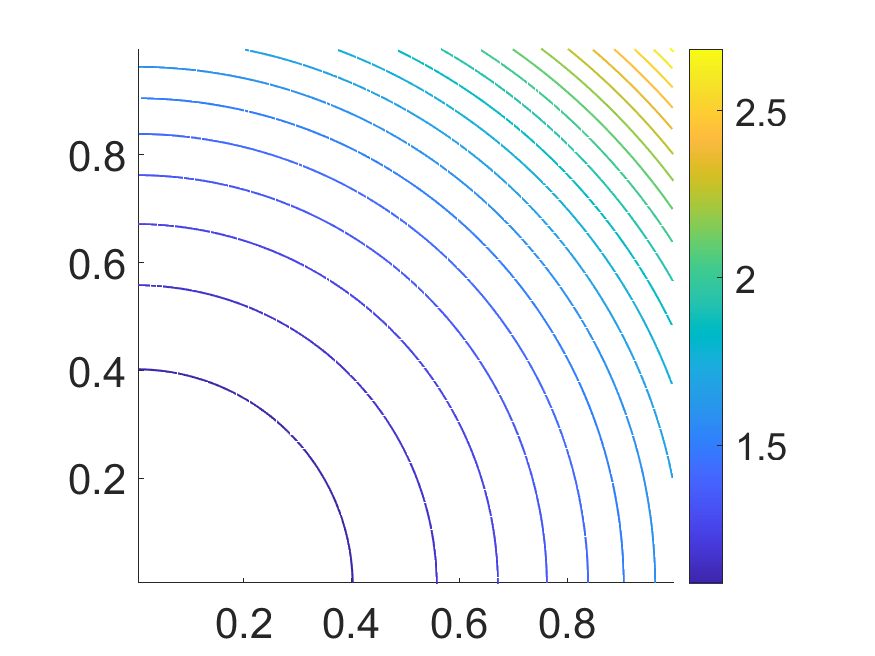}\\
	\end{tabular}
	\caption{(Monge‐Amp\`{e}re equation.) Numerical results for problem \eqref{expfunction} on the regular mesh. (a) Graph of the computed solutions for $h=1/80$. (b) Contours of the computed solutions for $h=1/80$.}
	\label{expfunction-graph}
\end{figure}

\begin{table}[t!]
	\centering
	(a)  
	\begin{tabular}{ccccccc}
		\hline
		$h$ &  DOS & Rate&FD&Rate & Proposed  & Rate\\
		\hline				
		1/20 &6.94$\times10^{-4}$ & & 9.61$\times10^{-5}$ &&6.69$\times10^{-5}$&\\
		1/40 & 1.92$\times10^{-4}$ &1.85&2.41$\times10^{-5}$& 2.00&
		1.68$\times10^{-5}$&1.99 \\    
		1/80&5.17$\times10^{-5}$&1.89 &6.03$\times10^{-6}$& 2.00&4.21$\times10^{-6}$&2.00 \\
		1/160&-&- &1.51$\times10^{-6}$& 2.00&1.05$\times10^{-6}$&2.00 \\
		\hline
	\end{tabular} \\ \vspace{0.1cm}
	(b)
	\begin{tabular}{ccccccc}
		\hline
		$h$ &  DOS &  Rate&FD&Rate & Proposed & Rate\\
		\hline					
		1/20 & 1.18$\times10^{-3}$ & & 1.71$\times10^{-4}$ &&1.20$\times10^{-4}$&\\
		1/40 & 3.86$\times10^{-4}$ &1.61& 4.29$\times10^{-5}$& 1.99&
		3.01$\times10^{-5}$&2.00 \\    
		1/80&1.27$\times10^{-4}$&1.60 &1.07$\times10^{-5}$& 2.00&7.55$\times10^{-6}$&2.00 \\
		1/160&-&- &2.69$\times10^{-6}$& 2.00&1.89$\times10^{-6}$&2.00 \\
		\hline
	\end{tabular} \\ \vspace{0.1cm}
	\caption{(Monge‐Amp\`{e}re equation.) Numerical results for problem \eqref{expfunction} on the regular mesh. (a) $L^2$ errors and convergence rates. (b) $L^{\infty}$ errors and convergence rates.}
	\label{expfunction-error}
\end{table}

\begin{table}[t!]
	\centering
	\begin{tabular}{ccccccccc}
		\hline
		$h$ &  1/20&1/40 & 1/80&1/160\\
		\hline				
		DOS & 9.8 & 121.2 &1638.8&-\\
		FD & 0.6 & 2.3&15.6&139.2\\ 
		Proposed&0.6&2.2& 8.4 &35.3\\
		\hline
	\end{tabular} \\ \vspace{0.1cm}
	\caption{(Monge‐Amp\`{e}re equation.) CPU time(s) for problem \eqref{expfunction} on the regular mesh.}
	\label{expfunction-time}
\end{table}  

Our numerical result with $h=1/80$ is visualized in Figure \ref{expfunction-graph}(a) with cross sections visualized in Figure \ref{expfunction-graph}(b). Our algorithm captures the convex solution based on the intrinsic formulation itself, without projecting the Hessian matrix to a semi-positive definite matrix, as used in DOS \cite{Liu2}.

We next compare the new method with both the DOS and FD methods, and we present the $L^2$ and $L^{\infty}$ errors of all methods in Table \ref{expfunction-error}. Both the proposed method and the FD method give a convergence rate of 2, so that they are superior to the DOS method in terms of convergence rate. Among all three methods, ours yields the smallest errors for both $L^2$ and $L^{\infty}$ norms and all mesh parameter $h$'s. To compare the computational cost, we present in Table \ref{expfunction-time} the CPU times used by all methods to obtain results in Table \ref{expfunction-error}. On a coarse mesh, such as $h=1/20$ and $1/40$, the CPU time of our method is comparable to that of FD and is less than that of DOS. As the mesh is refined, our method is faster than FD. 

Compared to the FD method, an advantage of the proposed algorithm is that it can be easily applied to solve problems on complex domains with irregular boundaries. Consider the triangulation of the domain 
$\Omega=\{(x_1,x_2)|(x_1-0.5)^2+(x_2-0.5)^2<1/4\}$,
as visualized in Figure \ref{mesh}(c); we test our new scheme against the DOS method on problem \eqref{expfunction} in this domain. 
We report both the $L^2$ and $L^{\infty}$ errors as well as corresponding convergence rates in Table \ref{expfunction-error-circle}.
Table \ref{expfunction-error-circle}(a) and Table \ref{expfunction-error-circle}(b) show that our proposed algorithm still performs better on the unstructured mesh with curved boundary than the DOS method, since our method provides smaller errors and higher convergence rates than DOS in terms of both $L^2$ and $L^{\infty}$ norms for all $h$'s.

\begin{table}[t!]
	\centering
	(a)  
	\begin{tabular}{ccccccc}
		\hline
		$h$ &  DOS & Rate & Proposed  & Rate\\
		\hline				
		1/10 &2.29$\times10^{-3}$& &7.29$\times10^{-4}$&\\
		1/20 & 8.03$\times10^{-4}$ &1.51 &
		1.69$\times10^{-4}$& 2.11\\    
		1/40&2.78$\times10^{-4}$& 1.53&2.94$\times10^{-5}$& 2.52\\
		1/80&8.16$\times10^{-5}$&1.77&8.26$\times10^{-6}$&1.83 \\
		\hline
	\end{tabular} \\ \vspace{0.1cm}
	(b)
	\begin{tabular}{ccccccc}
		\hline
		$h$ &  DOS &  Rate & Proposed & Rate\\
		\hline					
		1/10 & 5.21$\times10^{-3}$ & &2.44$\times10^{-3}$&\\
		1/20 & 3.00$\times10^{-3}$ &0.80&
		7.48$\times10^{-4}$&1.71\\    
		1/40&1.36$\times10^{-3}$&1.14 &1.74$\times10^{-4}$&2.10\\
		1/80&4.97$\times10^{-4}$&1.45 &4.79$\times10^{-5}$& 1.86\\
		\hline
	\end{tabular} \\ \vspace{0.1cm}
	\caption{(Monge‐Amp\`{e}re equation.) Numerical results for problem \eqref{expfunction} on a half-unit disk. (a) $L^2$ errors and convergence rates. (b) $L^{\infty}$ errors and convergence rates.}
	\label{expfunction-error-circle}
\end{table}   

\begin{figure}[t!]
	\centering
	\begin{tabular}{c}    		\includegraphics[width=0.4\textwidth]{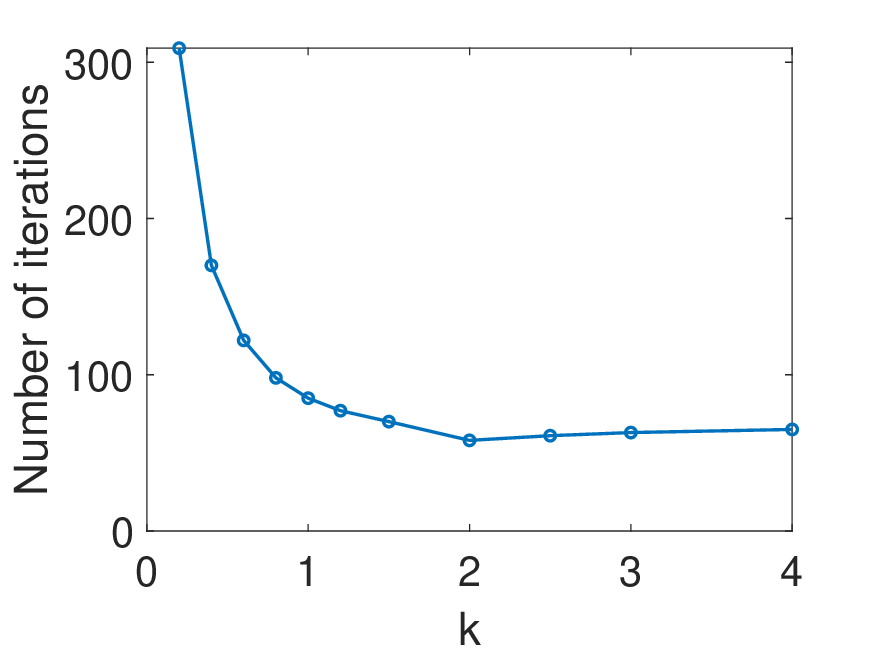}  \\
	\end{tabular}    
	\caption{(Monge-Amp\`ere equation.) Numerical results for problem \eqref{expfunction} with $\tau=h$ on the regular mesh of the unit square $(0, 1)^2$: the parameter $\gamma = k \hat{\lambda}_1/h^2$, where $k$ is a scaling factor.}
	\label{kt-ma}
\end{figure}

As in the semilinear case, we examine the influence of the choice of $\gamma$ on the algorithm for the Monge-Amp\`ere  equation with $\tau=h$. The numerical results in Figure \ref{kt-ma} are similar to those observed for the semilinear problem \eqref{Semilinear}, supporting $\gamma =\hat{\lambda}_1/h^2$ as an effective choice.

\subsubsection{An Obstacle Problem}
We consider an obstacle problem for the Monge-Amp\`ere equation given as
\begin{equation}\label{obstacle}
	\begin{cases}
		\det \Db^2u=f \ \mbox{in} \  \Omega,\\
		g=\dfrac{1}{2} \left(\max( |x - x_0| - 0.2,0)\right)^2 \  \mbox{on} \  \partial \Omega,
	\end{cases}
\end{equation}
where $\Omega = (0, 1)^2$, $f=\max( 1 - \frac{0.2}{|x - x_0|},0 )$, and $x_0 = (0.5,0.5)$. The exact solution $u$ is
$u=\frac{1}{2} (\max( |x - x_0| - 0.2,0))^2 \  \mbox{in} \   \Omega$,
which is convex and is in $C^1(\Omega)$. 
It is noted that the value of $u$ in equation \eqref{obstacle}
is zero within the open disk of radius $0.2$ centered at 
$(0.5,0.5)$, rendering problem \eqref{obstacle} degenerate elliptic. We consider problem \eqref{obstacle} as an obstacle problem for the Monge-Amp\`ere operator. 

We present our numerical solution on the regular mesh with $h=1/80$ in Figure \ref{obstacle-shoot}(a), whose level curves are visualized in Figure \ref{obstacle-shoot}(b), and the numerical solution is smooth and convex. In Figure \ref{obstacle-shoot}(b), we observe a region of constant values, corresponding to the region over which the right-hand side $f$ of problem (\ref{obstacle}) is zero.

Both FD and DOS have been applied to solve problem \eqref{obstacle} on the regular mesh. In \cite{Liu2}, the authors reported that DOS is divergent if it is directly applied to \eqref{obstacle}. To deal with this dilemma, they regularize the function $f$ by
$$
f_\eta=\max\left( 1 - \frac{0.2}{|x - x_0|},\eta \right) \ \mbox{in} \   \Omega,
$$
where $\eta=h$ or $h^2$.
When $\eta=h^2$, the accuracy of the DOS algorithm is quite high, but it needs a large number of iterations, leading to a huge computational cost.
On the other hand, our proposed algorithm can be directly applied to problem \eqref{obstacle} and the computational cost is much lower than that of DOS. It is worth noting that  $\tau=1$ can accelerate the convergence speed in the DOS algorithm in this particular example. 
\begin{figure}[t!]
	\centering
	\begin{tabular}{c@{\hspace{1.5cm}}c}
		(a) & (b)\\
		\includegraphics[width=0.35\textwidth]{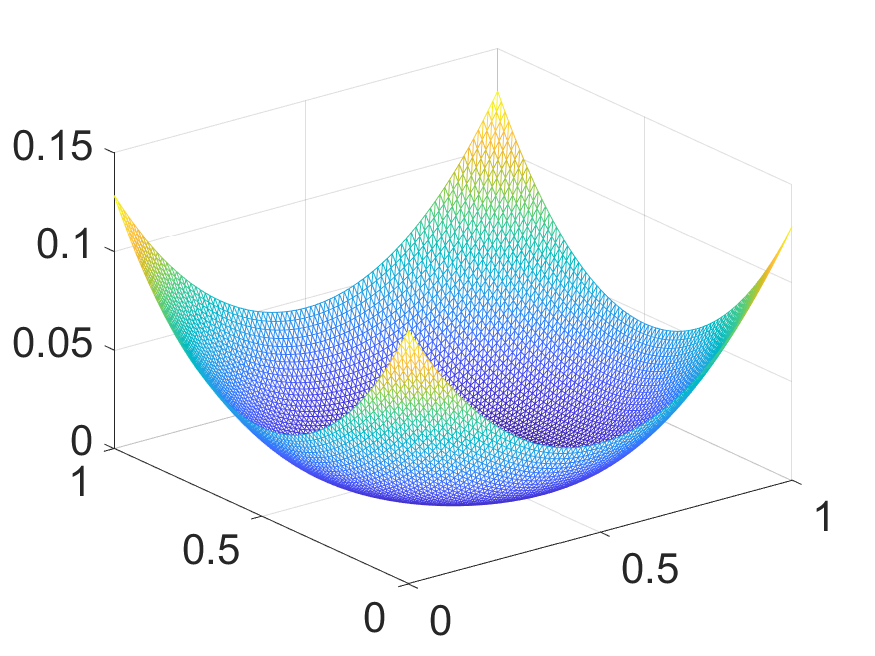} &
		\includegraphics[width=0.35\textwidth]{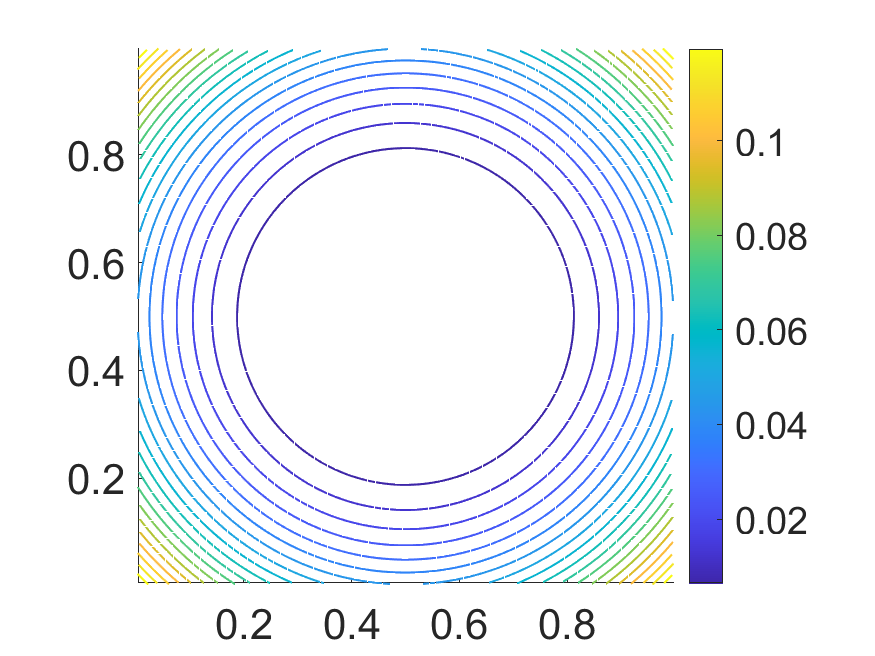}
	\end{tabular}
	\caption{(Monge‐Amp\`{e}re equation.) Numerical results for problem \eqref{obstacle} on the regular mesh. (a) Graph of the computed solutions for $h=1/80$. (b) Contours of the computed solutions for $h=1/80$.}
	\label{obstacle-shoot}
\end{figure}

\begin{table}[t!]
	\centering
	(a) 
	\begin{tabular}{ccccccccc}
		\hline
		$h$ &  DOS &  Rate&FD&Rate & Proposed  & Rate\\
		\hline					
		1/20 &4.14$\times10^{-4}$ & & 2.94$\times10^{-4}$ & &2.53$\times10^{-4}$&\\
		1/40 & 9.88$\times10^{-5}$ &2.07&1.08$\times10^{-4}$&1.44 &
		9.37$\times10^{-5}$& 1.43\\    
		1/80&3.92$\times10^{-5}$&1.33 &3.46$\times10^{-5}$& 1.64&2.98$\times10^{-5}$& 1.65\\   1/160&2.39$\times10^{-5}$&0.71&1.35$\times10^{-5}$& 1.36&1.17$\times10^{-5}$& 1.35\\
		1/320&-&-&5.14$\times10^{-6}$& 1.39&4.44$\times10^{-6}$& 1.40\\
		\hline
	\end{tabular} \\ \vspace{0.1cm}
	(b) 
	\begin{tabular}{ccccccccc}
		\hline
		$h$ &  DOS &  Rate&FD&Rate & Proposed  & Rate\\
		\hline					
		1/20 &1.09$\times10^{-3}$ & &6.70$\times10^{-4}$ &&5.90$\times10^{-4}$&\\
		1/40 & 3.13$\times10^{-4}$ &1.80& 2.71$\times10^{-4}$&1.30 &
		2.65$\times10^{-4}$&1.15 \\    
		1/80&1.85$\times10^{-4}$&0.76 &1.01$\times10^{-4}$& 1.42&1.06$\times10^{-4}$& 1.32\\
		1/160&1.17$\times10^{-4}$&0.66 &4.41$\times10^{-5}$& 1.20&4.79$\times10^{-5}$& 1.15\\
		1/320&-&-&1.85$\times10^{-5}$& 1.25&2.10$\times10^{-5}$& 1.19\\
		\hline
	\end{tabular} \\ \vspace{0.1cm}
	\caption{(Monge‐Amp\`{e}re equation.) Numerical results for problem \eqref{obstacle} on the regular mesh. (a) $L^2$ errors and convergence rates. (b) $L^{\infty}$ errors and convergence rates.}
	\label{obstacle-error}
\end{table}
\begin{table}[t!]
	\begin{center}
		\begin{tabular}{ccccccccc}
			\hline
			$h$ &  1/20&1/40& 1/80&1/160&1/320\\
			\hline				
			DOS & 1.4 & 7.5 &69.5&1846.4 &-\\
			FD &  0.4 & 1.5&9.5& 73.2& 562.5\\ 
			Proposed& 1.0&3.3& 14.2&72.4 &216.6 \\
			\hline
		\end{tabular} \\ \vspace{0.1cm}
	\end{center}
	\caption{(Monge‐Amp\`{e}re equation.) CPU time(s) for problem \eqref{obstacle} on the regular mesh.}
	\label{obstacle-time}
\end{table}	
The comparison of the proposed algorithm with DOS and FD is shown in Table \ref{obstacle-error}. In this example, the stopping criterion of our proposed algorithm is $||u^{n+1}-u^n||<10^{-10}$ for $h=1/320$. 
We observe that the convergence rates of our proposed algorithm exceed $1$ in both $L^2$ and $L^\infty$ norms.
Our proposed algorithm demonstrates performance comparable to that of FD but outperforms that of DOS. In terms of efficiency, we compare the CPU times of all three methods in Table \ref{obstacle-time}. On very coarse meshes, such as $h=1/20$, $1/40$, and $1/80$, FD is the most efficient one; however, on a finer mesh such as $h=1/320$, the proposed algorithm is much faster than FD.

\subsubsection{An Example with Singular Solution}
In this example, we consider the following problem
\begin{equation}\label{singularity-b}
	\begin{cases}
		\det \textbf{D}^2u=\dfrac{4}{(1-4r^2)^2}\;\  \mbox{in} \;\ \Omega,\\
		g=0\;\ \mbox{on}\;\ \partial \Omega,
	\end{cases}
\end{equation}
where $r=\sqrt{(x_1-0.5)^2+(x_2-0.5)^2}$, and $\Omega=\{(x_1,x_2)|(x_1-0.5)^2+(x_2-0.5)^2<1/4\}$ is a half-unit disk which is triangulated in Figure \ref{mesh}(c). Then problem \eqref{singularity-b} has a strictly convex solution $u$ which is given by
\begin{align*}
	u=-\frac{1}{2}\sqrt{1-4r^2} \;\ \mbox{in} \;\ \Omega.
\end{align*}	 
The solution $u$ satisfies that $u \in C^0(\overline{\Omega}) \cap W^{1,s}(\Omega),\; \forall s \in  [1,2)$. However, we note that the function $u$ is not as smooth as those solutions in previous examples, because the value of $|\nabla u|$ is infinite on the boundary of $\Omega$. Consequently, problem \eqref{singularity-b} is a good example to test robustness of our algorithm. 

\begin{table}[!ht]
	\begin{center}
		\begin{tabular}{cccccc}
			\hline
			$h$& $L^2$ error & Rate  &  $L^{\infty}$ error & Rate\\
			\hline							
			1/20&6.59$\times10^{-2}$ &   & 8.29$\times10^{-2}$ & \\
			1/40&4.10$\times10^{-2}$ & 0.68  &  5.92$\times10^{-2}$&0.49\\
			1/80&2.18$\times10^{-2}$ & 0.91  & 4.28$\times10^{-2}$&0.47\\
			1/160&8.23$\times10^{-3}$ & 1.41  &  3.10$\times10^{-2}$ & 0.47\\
			
			\hline
		\end{tabular} \\ \vspace{0.1cm}
	\end{center}
	\caption{(Monge‐Amp\`{e}re equation.) Numerical results by our proposed algorithm for problem \eqref{singularity-b} on a half-unit disk: $L^2$ errors, $L^{\infty}$ errors, and corresponding convergence rates.}
	\label{test-singularity-b}
\end{table}

\begin{figure}[!ht]
	\centering
	\begin{tabular}{cc@{\hspace{1.5cm}}cc}	
		\multicolumn{3}{c}{(a)} \\ 
		\includegraphics[width=0.35\textwidth]{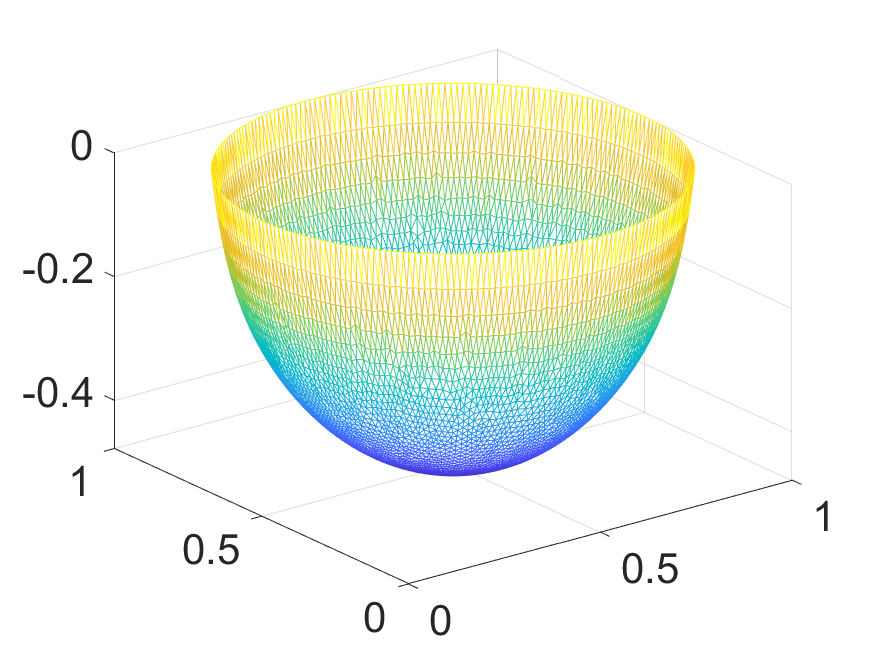} &  &
		\includegraphics[width=0.35\textwidth] 
		{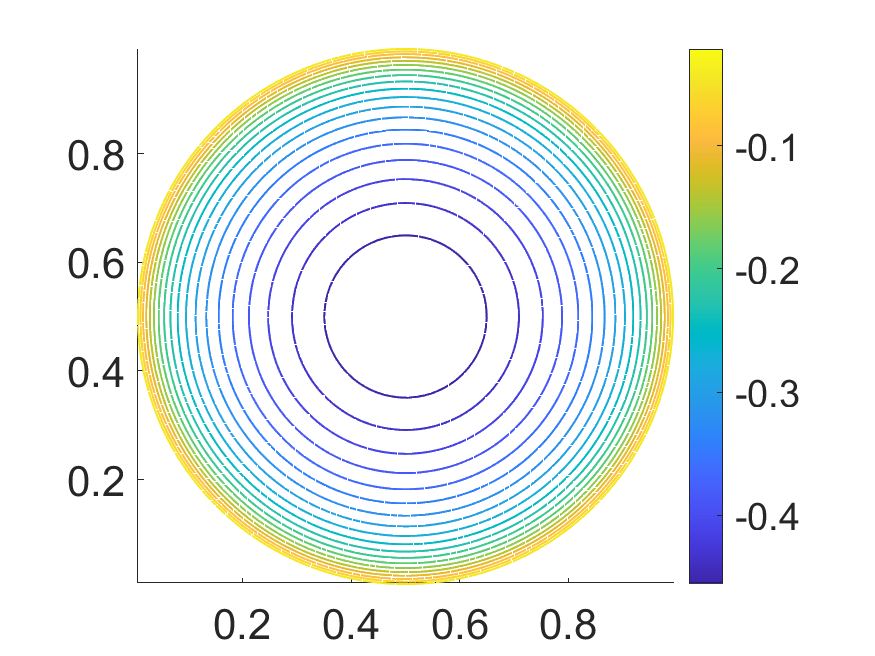}\\        \multicolumn{3}{c}{(b)} \\ 
		\includegraphics[width=0.35\textwidth]{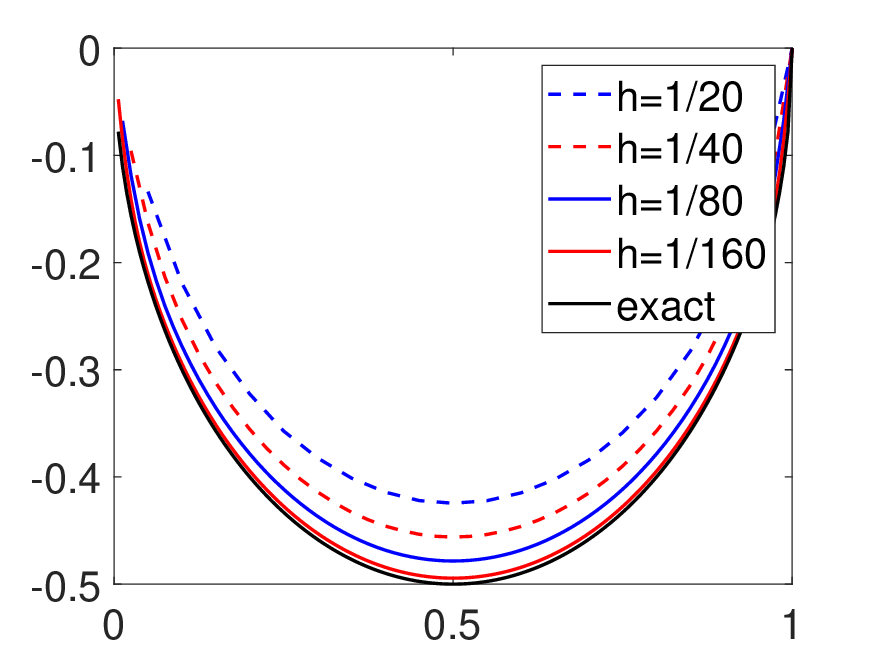} &  &
		\includegraphics[width=0.35\textwidth] 
		{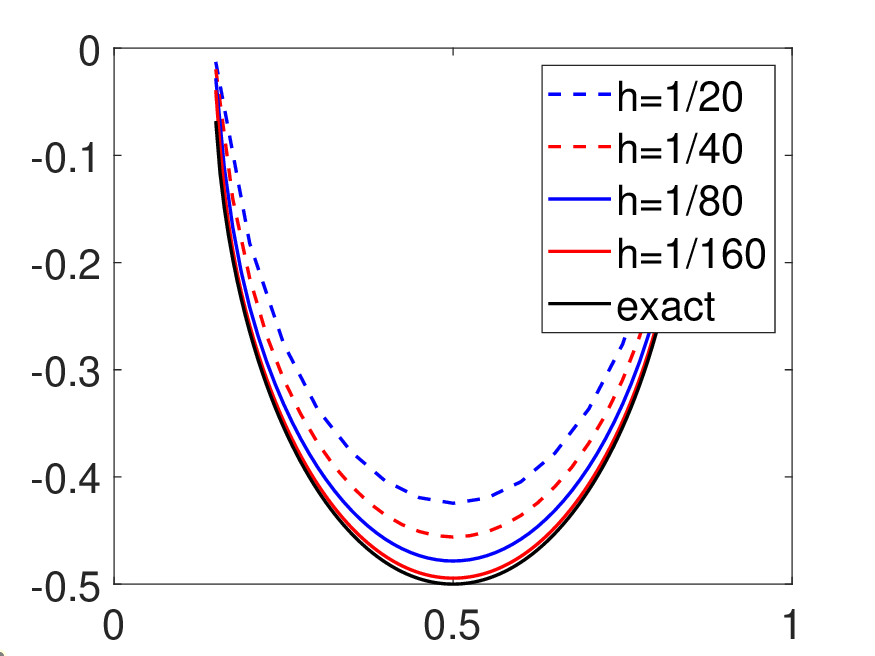}\\         
	\end{tabular}	
	\caption{(Monge‐Amp\`{e}re equation.) Numerical results for problem \eqref{singularity-b} on the unstructured mesh of the half-unit disk. (a) Graphs and contours of the computed solution. (b) Cross sections of the computed results along the line $x_2 = 1/2$ (left)
		and the line $x_1 = x_2$ (right) for $h = 1/20$, $1/40$, $1/80$, and $1/160$, respectively.}
	\label{singularity-b-graph}
\end{figure}

Numerical results by the proposed method are reported in Table \ref{test-singularity-b} and Figure \ref{singularity-b-graph}. 
From Table \ref{test-singularity-b}, we observe that the convergence orders are approximately $1$ and $0.5$, respectively, in terms of the $L^2$ and $L^{\infty}$ norm. Figure \ref{singularity-b-graph} shows that our new method is able to capture very well the convex solution with singularity on the  boundary.

\subsubsection{An Example without Classical Solution}
In this experiment we consider a Monge-Amp\`ere equation without an exact solution:
\begin{equation}\label{no-solution}
	\begin{cases}
		\det \textbf{D}^2u=1\ \mbox{in} \ \Omega,\\
		u=0\ \mbox{on}\ \partial \Omega,
	\end{cases}
\end{equation}
where $\Omega$ is the unit square $(0, 1)^2$.
Problem \eqref{no-solution} does not have a classical solution, but it admits a viscosity solution. How to compute the viscosity solution for this problem has been studied in \cite{Liu2, Froese2}, and we compare our results with theirs to validate that our new algorithm is able to compute the generalized solution as well.

Since no exact solution is available for comparison, we focus on checking the minimum value of the computed solution. Numerical results and corresponding graphs are presented in Table \ref{test-no-solution} and Figure \ref{test-graph-no-solution}, respectively.
Table \ref{test-no-solution} indicates that the minimum value of the solution obtained by our algorithm is $-0.1826$  for $h = 1/80$, which is close to $-0.182625$ and $-0.1831$ as reported in \cite{Froese2} and \cite{Liu2}, respectively. Our algorithm is slower for problem \eqref{no-solution} compared to the previous examples in terms of the number of iterations, because $f$ in this example is very small, and it is observed from many numerical experiments that a relatively large $f$ leads to a relatively fast convergence behavior.

\begin{table}[!ht]
	\centering
	\begin{tabular}{ccccc}
		\hline
		$h$&1/10&1/20&1/40&1/80\\
		\hline		
		Mini value& -0.1615 & -0.1714 & -0.1786&-0.1826\\					
		Iterations&18 & 33& 65&126 \\
		\hline
	\end{tabular} \\ \vspace{0.1cm}
	\caption{(Monge‐Amp\`{e}re equation.) Numerical results for problem \eqref{no-solution} on the unstructured mesh as shown in Figure \ref{mesh}(b): the number of iterations and the minimum value.}
	\label{test-no-solution}
\end{table}     
\begin{figure}[!ht]
	\centering
	\begin{tabular}{c@{\hspace{1.5cm}}c} 
		\multicolumn{2}{c}{(a)} \\
		\includegraphics[width=0.35\textwidth]{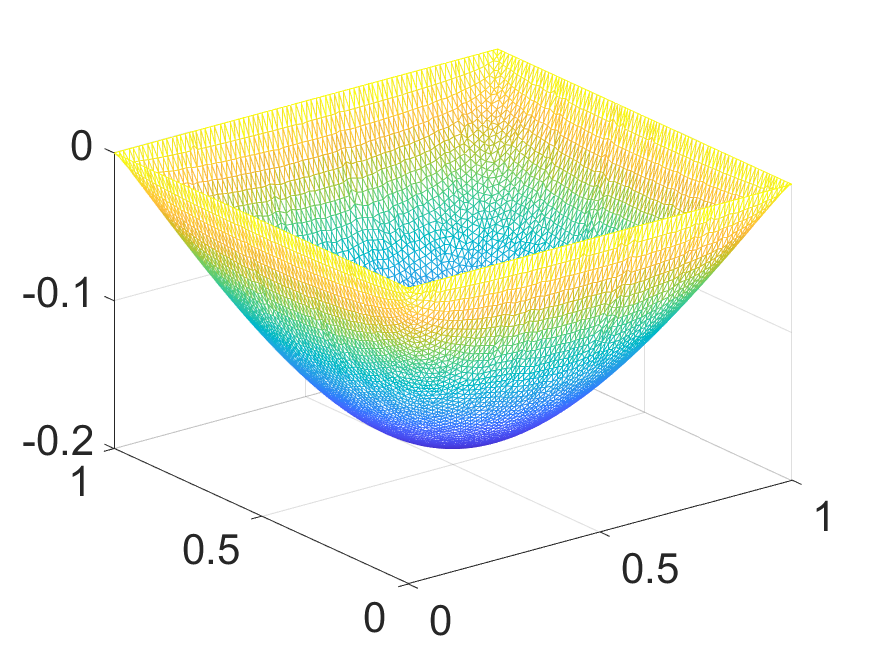} &  
		\includegraphics[width=0.35\textwidth]{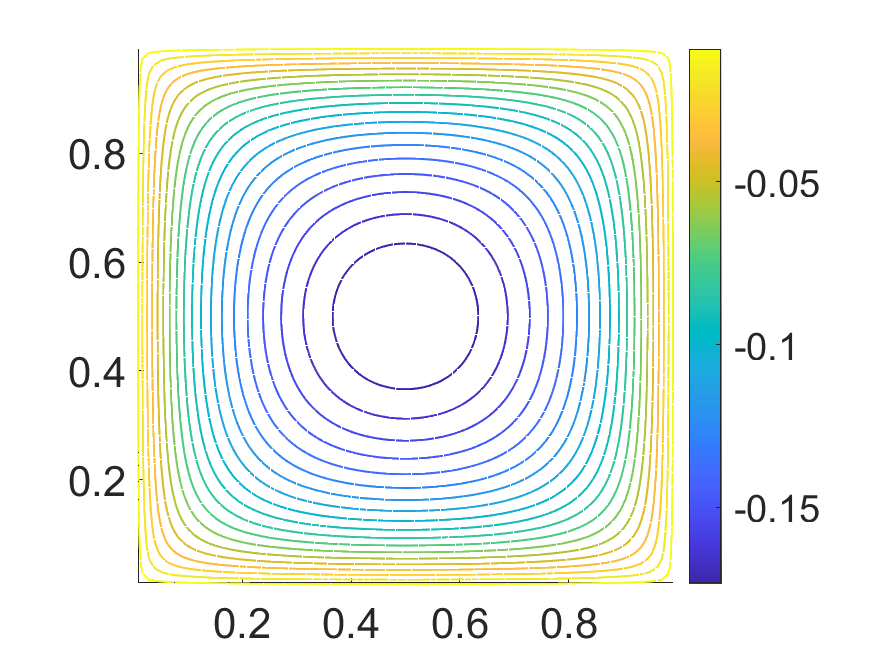} \\
		\multicolumn{2}{c}{(b)} \\
		\includegraphics[width=0.35\textwidth]{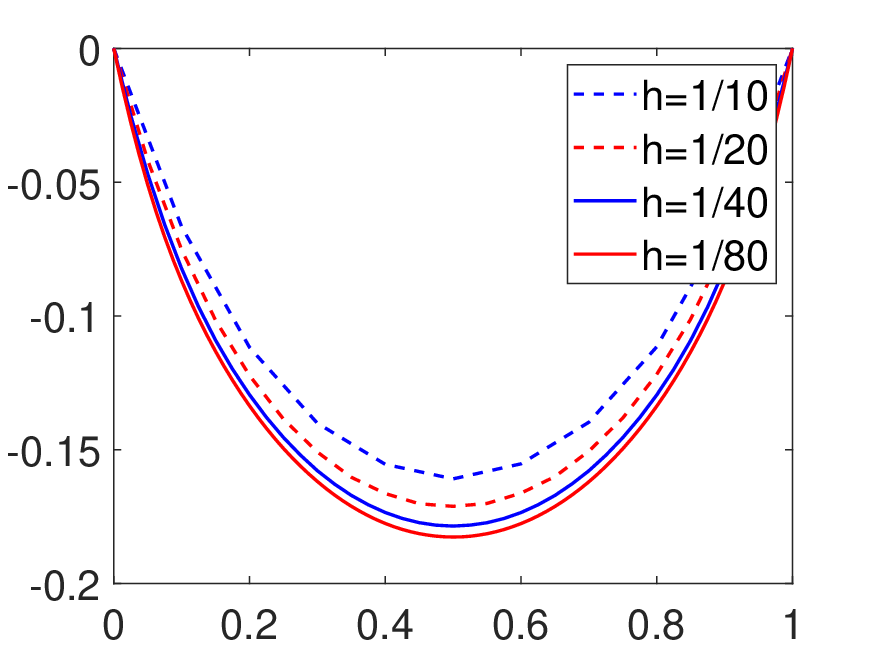} &  
		\includegraphics[width=0.35\textwidth]{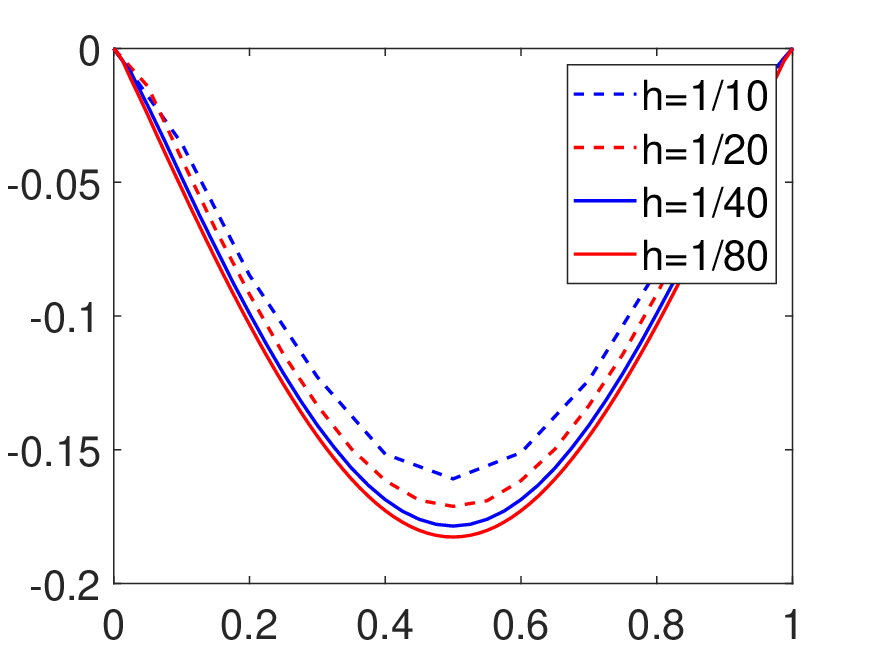} \\
	\end{tabular}    
	\caption{(Monge‐Amp\`{e}re equation.) Numerical results for problem \eqref{no-solution} on the unstructured mesh as shown in Figure \ref{mesh}(b). (a) Graph and contours of the computed solution. (b) Cross sections of computed results along the line $x_1 = 1/2$ (left)
		and the line $x_1 = x_2$ (right) for $h = 1/10$, $1/20$, $1/40$, and $1/80$, respectively.}
	\label{test-graph-no-solution}
\end{figure}

The test problem \eqref{no-solution} demonstrates that the non-strict convexity of $[0,1]^2$ results in the non-existence of a smooth solution. To study the effect of boundary corners, we transform the unit square into the following strictly convex domain, defined by
\begin{equation}\label{eye}
	\Omega = \{(x_1, x_2) \mid -x_1(1-x_1) < x_2 < x_1(1-x_1), 0 < x_1 < 1\}.
\end{equation}
The triangulation of this domain is visualized in Figure \ref{eye-graph}. It is shown in \cite{Liu2} that problem (\ref{no-solution}) on domain (\ref{eye}) does not have a classical solution.
%The proof of non-existence of smooth solution in the domain \eqref{eye} can be found in \cite{Liu2}. 

Figure \ref{test-eye-graph} presents us with very detailed information about the solution and its properties. Except for $\{0, 0\}$ and $\{1, 0\}$, the value of $|\nabla u|$ approach infinity on the entire boundary. The minimum value is $-0.0538$ for $h=1/80$, which is consistent with the result in \cite{Liu2}. 

\begin{figure}[!ht]
	\centering
	\begin{tabular}{cc}    
		\includegraphics[clip, trim = {30 10 30 10}, width=0.4\textwidth]{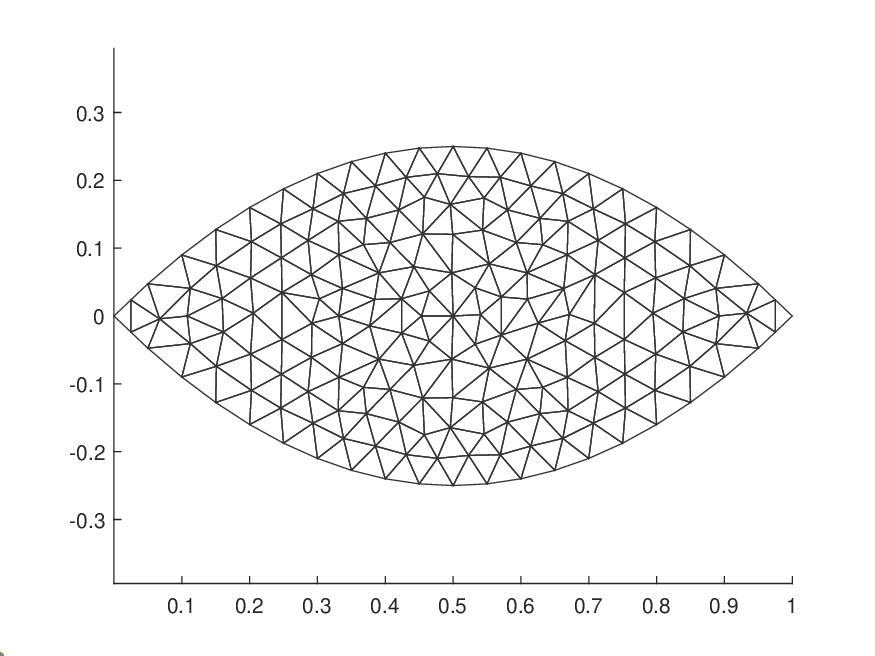} 
	\end{tabular}    
	\caption{A triangulation of the eye-shaped domain}
	\label{eye-graph}
\end{figure}   

\begin{table}[!ht]
	\centering
	\begin{tabular}{ccccc}
		\hline
		$h$&1/10&1/20&1/40&1/80\\
		\hline		
		Mini value& -0.0515 & -0.0528 & -0.0533&-0.0538\\					
		Iterations&15 & 15& 18&23 \\
		\hline
	\end{tabular} \\ \vspace{0.1cm}
	\caption{(Monge‐Amp\`{e}re equation.) Numerical results for problem  \eqref{no-solution} on the eye-shaped  domain: the number of iterations and the minimum value.}
	\label{test-eye}
\end{table}

\begin{figure}[!ht]
	\centering
	\begin{tabular}{c@{\hspace{1.5cm}}c}  
		\multicolumn{2}{c}{(a)} \\
		\includegraphics[width=0.35\textwidth]{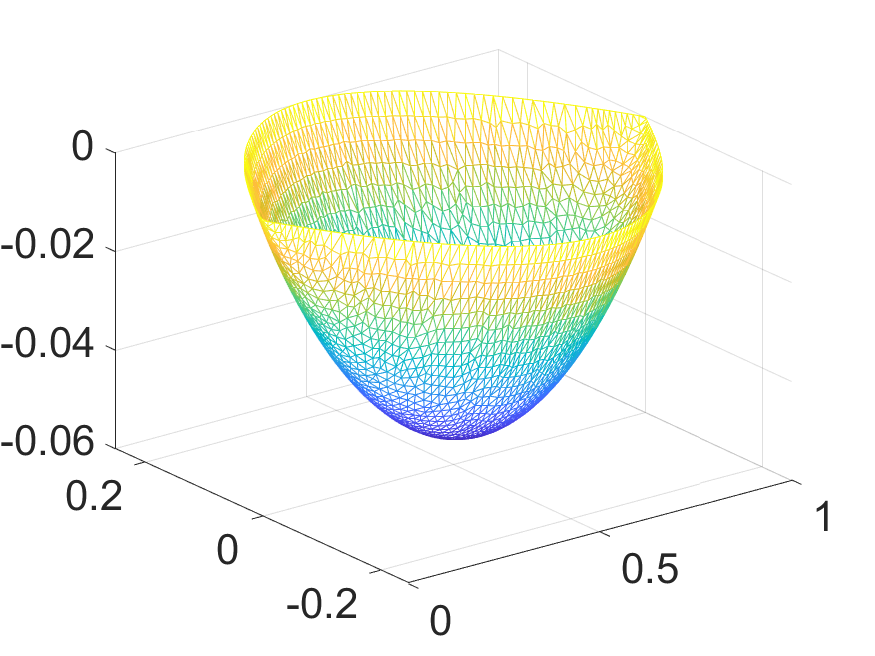} &  
		\includegraphics[width=0.35\textwidth]{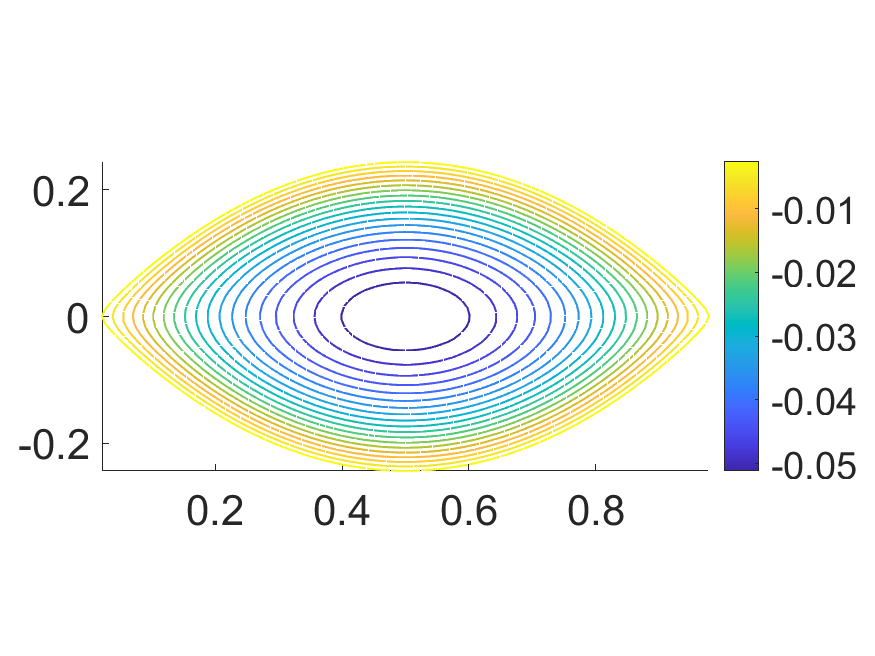} \\
		\multicolumn{2}{c}{(b)} \\
		\includegraphics[width=0.35\textwidth]{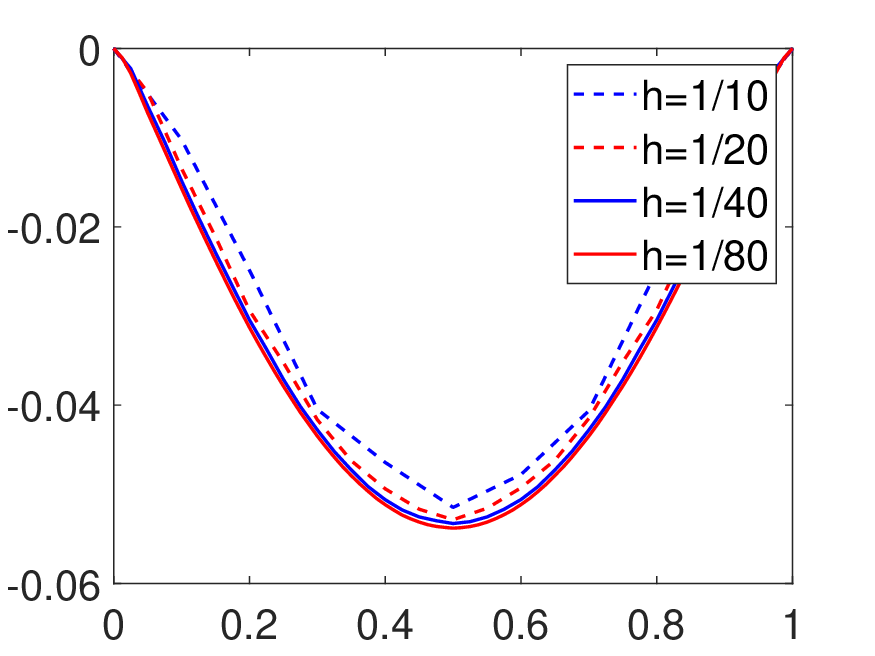} &  
		\includegraphics[width=0.35\textwidth]{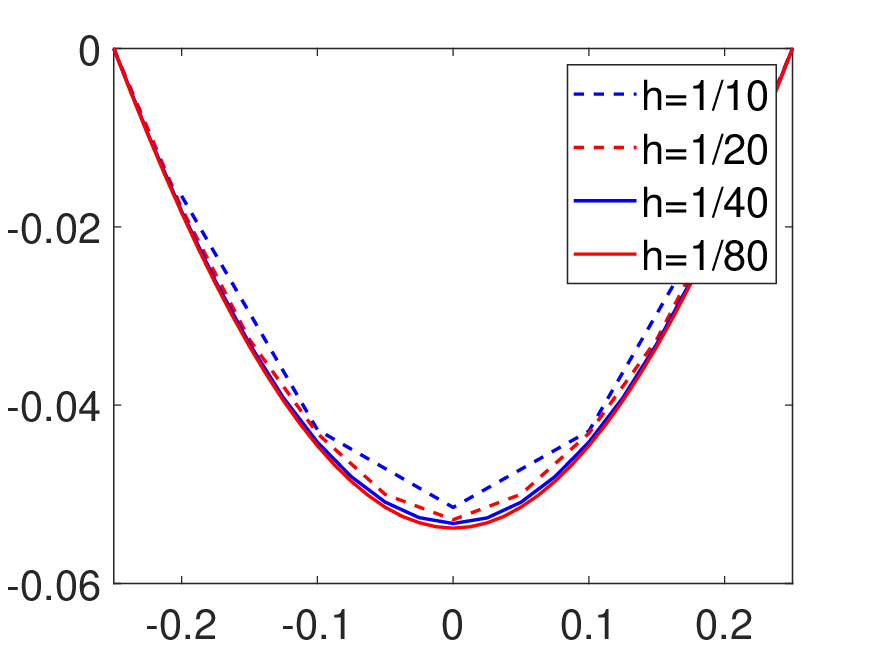} \\
	\end{tabular}    
	\caption{(Monge‐Amp\`{e}re equation.) Numerical results for problem \eqref{no-solution} on the eye-shaped  domain. (a) Graph and contours of the computed solution;  (b) Cross sections of the computed results along the line $x_2 = 0$ (left) and the line $x_1 = 1/2$ (right) for $h = 1/10$, $1/20$, $1/40$, and $1/80$, respectively.}
	\label{test-eye-graph}
\end{figure}

\subsection{Pucci's Equation}
As for the Pucci's equation (\ref{pucci}), we solve two problems to  compare our proposed algorithm with the least-squares (LS) method designed by Caffarelli and Glowinski \cite{Glowinski}.

\subsubsection{Smooth Solution}
In this experiment, we consider the boundary condition defined by
\begin{equation}\label{P-smooth}
	\begin{aligned}
		g(x) = -\rho^{1-\alpha} \  \mbox{on} \  \partial \Omega,
	\end{aligned}
\end{equation}
where $\rho =\sqrt{(x_1+1)^2+(x_2+1)^2}$ and $\Omega=(0, 1)^2$. 
As shown in \cite{Glowinski}, the exact solution is given by 
$u = -\rho^{1-\alpha} \ \mbox{in} \ \Omega$.
Since $(-1,-1)\notin\overline{\Omega}$, $u \in C^{\infty}(\overline{\Omega})$ is a smooth solution of the Pucci's equation for $\alpha>1$.	
By testing this example, we have two primary goals: (i) investigate the convergence rate of our algorithm for the Pucci's equation; (ii) examine the influence of the parameter $\alpha$ on the performance of our proposed algorithm. 

In Table \ref{P-smooth-table1}, we present the numerical results obtained by our proposed algorithm, where the unstructured mesh in Figure \ref{mesh}(b) is used, and $\alpha$ takes three values:  \(\alpha = 2, 3, 4\), respectively. Both the \(L^2\) and the 
\(L^{\infty}\) errors demonstrate that our scheme achieves nearly optimal convergence rates.

By transforming the Pucci's equation into the Monge-Amp\`ere type equation
$
\alpha |\Delta u|^2 + (\alpha - 1)^2 \det \mathbf{D}^2 u = 0,
$
we observe that as \(\alpha\) increases, specifically when \(\alpha > \frac{3 + \sqrt{5}}{2}\), the Monge-Amp\`ere operator \(\det \mathbf{D}^2 u\) becomes relatively more significant than the squared Laplace-operator part $|\Delta u|^2$, making the problem more complicated. This feature is illustrated in Table \ref{P-smooth-table1}; as \(\alpha\) increases from \(2\) to \(4\), the number of iterations increases from \(11\) to \(21\) and the accuracy decreases. 

Additionally, we apply both our proposed algorithm and the LS method in \cite{Glowinski} to solve  the problem, where the regular mesh in Figure \ref{mesh}(a) is used to discretize the domain, and we further compare the results obtained by the two methods. As shown in Table \ref{comparision-p-smooth}, our new scheme has an optimal convergence rate, analogous to the least-squares method, but ours achieves higher accuracy.
\begin{table}[t!]
	\centering
	(a)
	\begin{tabular}{cccccccc}
		\hline
		$\alpha$&$h$&Iterations& $L^2$ error & Rate  &  $L^{\infty}$ error & Rate& CPU time\\
		\hline
		2&1/10&9 &6.71$\times10^{-5}$ &   &  1.77$\times10^{-4}$ & &0.3\\
		
		2&1/20&10 &1.55$\times10^{-5}$ & 2.11  & 5.88$\times10^{-5}$ & 1.59 &0.8\\
		2&1/40&11 &4.41$\times10^{-6}$ & 1.81   &  1.67$\times10^{-5}$&1.81&3.0\\
		2&1/80&11 &1.46$\times10^{-6}$ &  1.59 & 5.00$\times10^{-6}$&1.74&15.5\\
		\hline
	\end{tabular} \\ \vspace{0.1cm}
	(b)
	\begin{tabular}{cccccccc}
		\hline
		$\alpha$&$h$&Iterations& $L^2$ error & Rate  &  $L^{\infty}$ error & Rate& CPU time\\
		\hline		
		3&1/10&11 &2.13$\times10^{-4}$ &   &  4.59$\times10^{-4}$ & &0.3\\
		
		3&1/20&13 &5.88$\times10^{-5}$ & 1.86  & 1.54$\times10^{-4}$ &1.58& 0.8\\
		3&1/40&15 &1.70$\times10^{-5}$ & 1.79   &  4.65$\times10^{-5}$&1.73&3.4\\
		3&1/80&16 &5.17$\times10^{-6}$ &  1.72 & 1.38$\times10^{-5}$&1.75&17.5\\
		\hline
	\end{tabular} \\ \vspace{0.1cm}
	(c)
	\begin{tabular}{cccccccc}
		\hline
		$\alpha$&$h$&Iterations& $L^2$ error & Rate  &  $L^{\infty}$ error & Rate& CPU time\\
		\hline		
		4&1/10&12 &3.72$\times10^{-4}$ &   &  7.84$\times10^{-4}$ & &0.3\\
		
		4&1/20&15 &1.11$\times10^{-4}$ & 1.74  & 2.55$\times10^{-4}$ & 1.62&0.9\\
		4&1/40&19 &3.21$\times10^{-5}$ &  1.79  &  7.74$\times10^{-5}$&1.72&3.9\\
		4&1/80&21 &9.26$\times10^{-6}$ &  1.79 & 2.34$\times10^{-5}$&1.73&20.4\\
		\hline
	\end{tabular} \\ \vspace{0.1cm}
	\caption{(The Pucci's equation.) Numerical results for problem \eqref{P-smooth} on the unstructured mesh as shown in Figure \ref{mesh}(b). Number of iterations, numerical errors, convergence rates, and CPU time(s) of (a) $\alpha=2$, (b) $\alpha=3$, and (c) $\alpha=4$, respectively.}
	\label{P-smooth-table1}
\end{table}
\begin{table}[t!]
	\centering
	\begin{tabular}{cccccccc}
		\hline
		$\alpha$&$h$& LS & Rate  & Proposed & Rate\\
		\hline
		2&1/64&3.37$\times10^{-6}$ &   & 1.19$\times10^{-6}$& \\			
		2&1/128& 8.44$\times10^{-7}$&2.00 & 2.98$\times10^{-7}$  &  2.00 \\
		3&1/32&1.03$\times10^{-4}$ & &1.28$\times10^{-5}$   &  \\			
		3&1/64& 2.57$\times10^{-5}$& 2.00& 3.21$\times10^{-6}$  &  2.00 \\
		\hline
	\end{tabular} \\ \vspace{0.1cm}
	\caption{(The Pucci's equation.) Numerical results for problem \eqref{P-smooth} on the regular mesh as shown in Figure \ref{mesh}(a): $L^2$ errors and convergence rates of both LS \cite{Glowinski} and our proposed algorithm.}
	\label{comparision-p-smooth}
\end{table}

\subsubsection{Regularization of Boundary Data}
We further consider the Pucci's equation with following boundary condition: $g(x)$ is the characteristic function of a subset $\partial \Omega \setminus\bigcup_{i=1}^4 \varGamma_i$ of $\partial \Omega$, where 

\begin{equation}\label{P-nonsmooth}
	\begin{aligned}
		\varGamma_1 = \left\lbrace  x |  1/4 < x_1 < 3/4, x_2 = 0\right\rbrace,\
		\varGamma_2 = \left\lbrace x |   x_1 = 1, 1/4 < x_2 < 3/4\right\rbrace, \\
		\varGamma_3 = \left\lbrace x |  1/4 < x_1 < 3/4, x_2 = 1\right\rbrace,\
		\varGamma_4 = \left\lbrace x |  x_1 = 0, 1/4 < x_2 < 3/4\right\rbrace.\\ 
	\end{aligned}
\end{equation}
The function $g \notin H^{3/2}(\partial \Omega)$ implies that there is no solution in $H^2(\Omega)$. 

To compute the numerical solution, we approximate $g(x)$ by a smooth function, $g_\delta(x)$, 
which is defined in the following way. On $\varGamma_1$, we take the function $g_\delta(x)$ to be 	
\begin{align*}
	g_\delta(x)=
	\begin{cases}
		1, &0 \leq x_1 \leq 1/4-\delta \ \mbox{or} \ 3/4+\delta \leq x_1 \leq 1, \\
		\dfrac{1}{2}[1-\sin(\dfrac{\pi}{2}(x_1-\dfrac{1}{4})/\delta)],\quad &1/4-\delta \leq x_1 \leq 1/4+\delta,\\
		0,\quad &1/4+\delta \leq x_1 \leq 3/4-\delta,\\
		\dfrac{1}{2}[1+\sin(\dfrac{\pi}{2}(x_1-\dfrac{3}{4})/\delta)],\quad &3/4-\delta \leq x_1 \leq 3/4+\delta,\\
	\end{cases}
\end{align*} 
where the regularization parameter $\delta$ is taken to be $\delta=1/16$. $g_\delta(x)$ on other $\varGamma_i$ for $i=2,3,4$ are similarly defined. 

\begin{figure}[t!]
	\centering
	\begin{tabular}{ccc}
		(a)&(b)&(c)\\
		\includegraphics[width=0.3\textwidth]{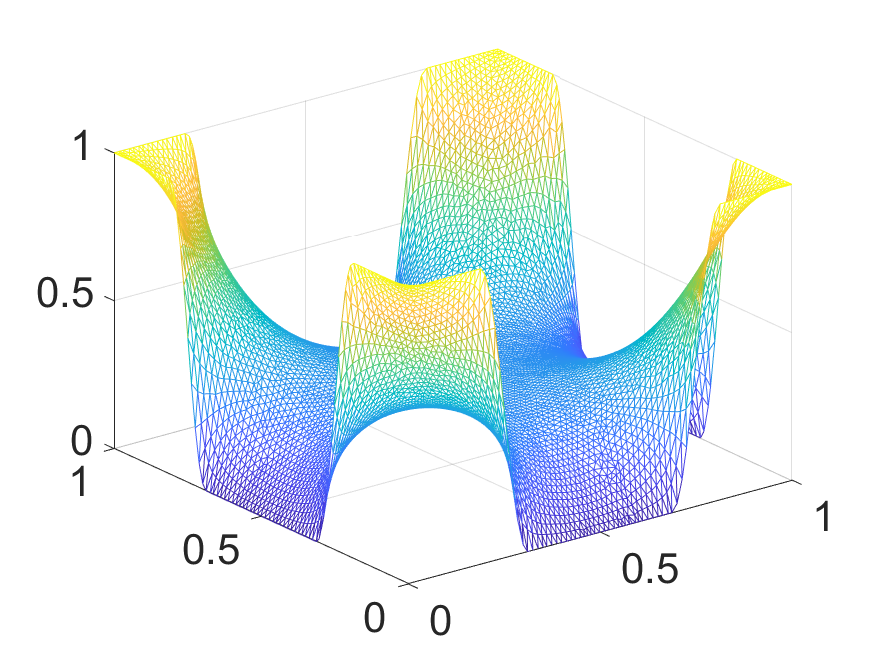} &
		\includegraphics[width=0.3\textwidth]{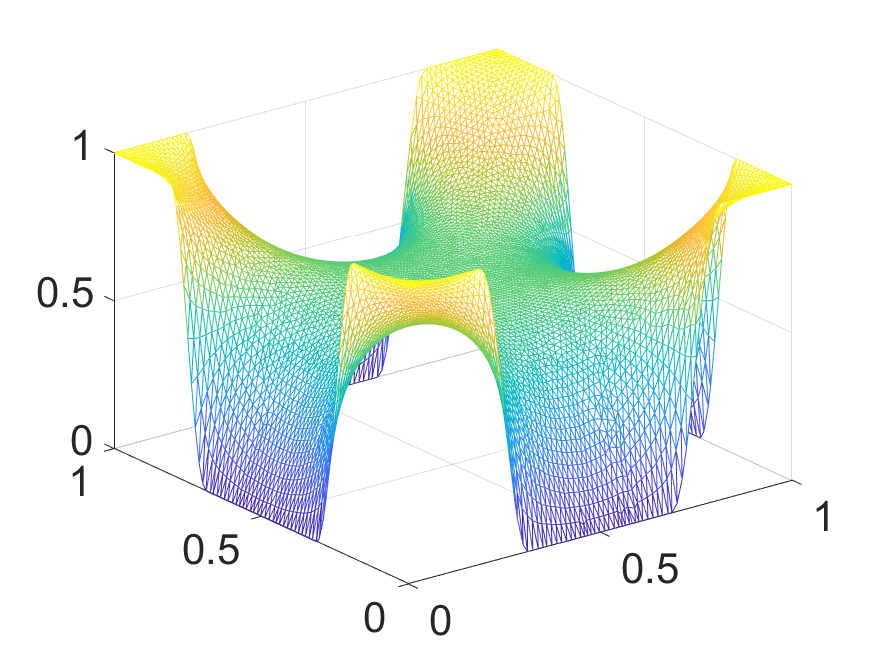} &
		\includegraphics[width=0.3\textwidth]{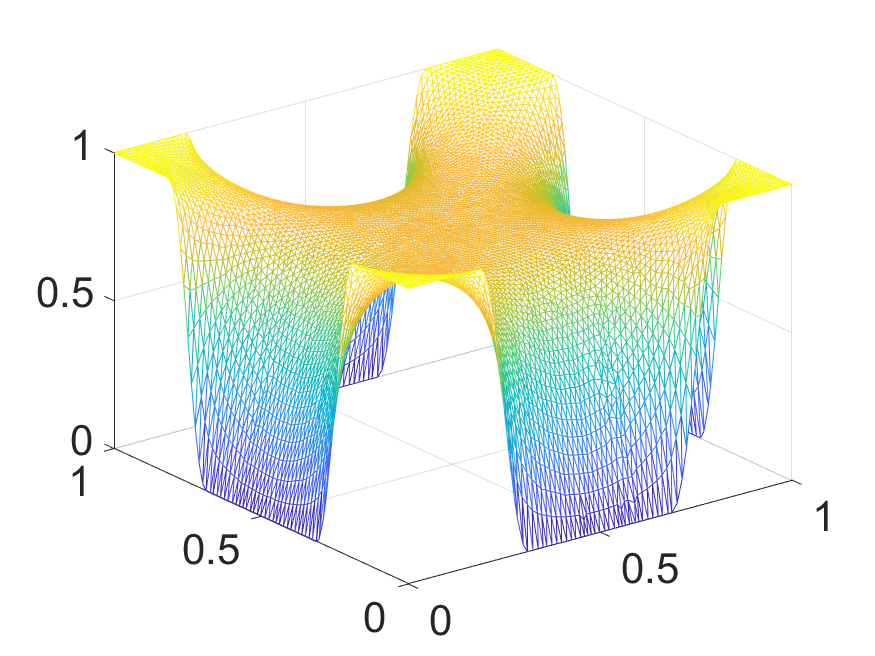} \\
	\end{tabular}
	\caption{(Pucci's equation.) Numerical results for problem \eqref{P-nonsmooth} on the unstructured mesh, Figure \ref{mesh}(b): The graph of the numerical solution with $h=1/80$ of (a) $\alpha=1.1$, (b) $\alpha=2.0$, and (c) $\alpha=3.0$, respectively.}
	\label{P-nonsmooth-graph}
\end{figure}

\begin{figure}[t!]
	\centering
	\begin{tabular}{cccccc}
		\multicolumn{1}{c}{($\alpha=1.1$)} & \multicolumn{1}{c}{($\alpha=2.0$)} & \multicolumn{1}{c}{($\alpha=3.0$)} \\
		\includegraphics[width=0.3\textwidth]{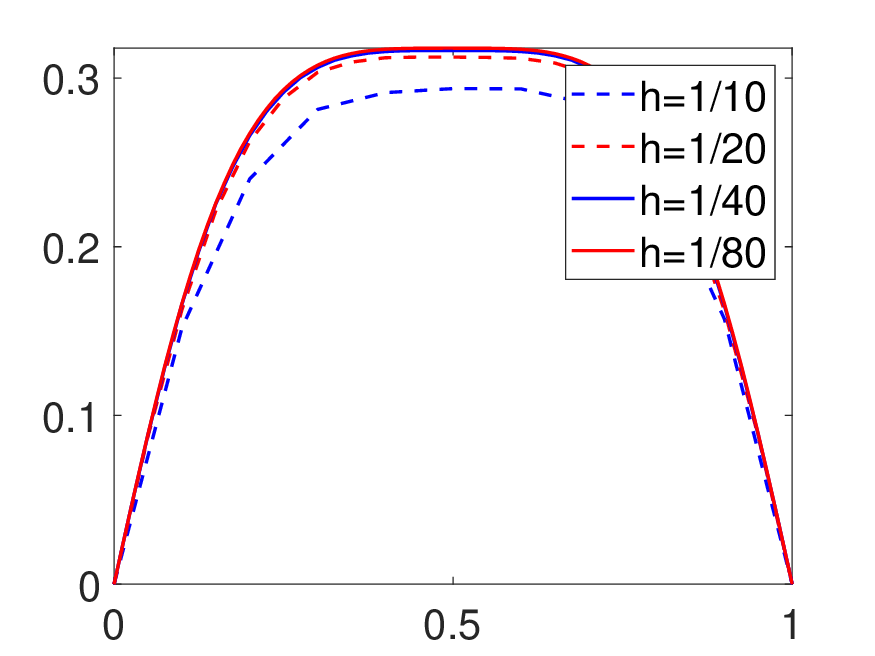} &
		\includegraphics[width=0.3\textwidth]{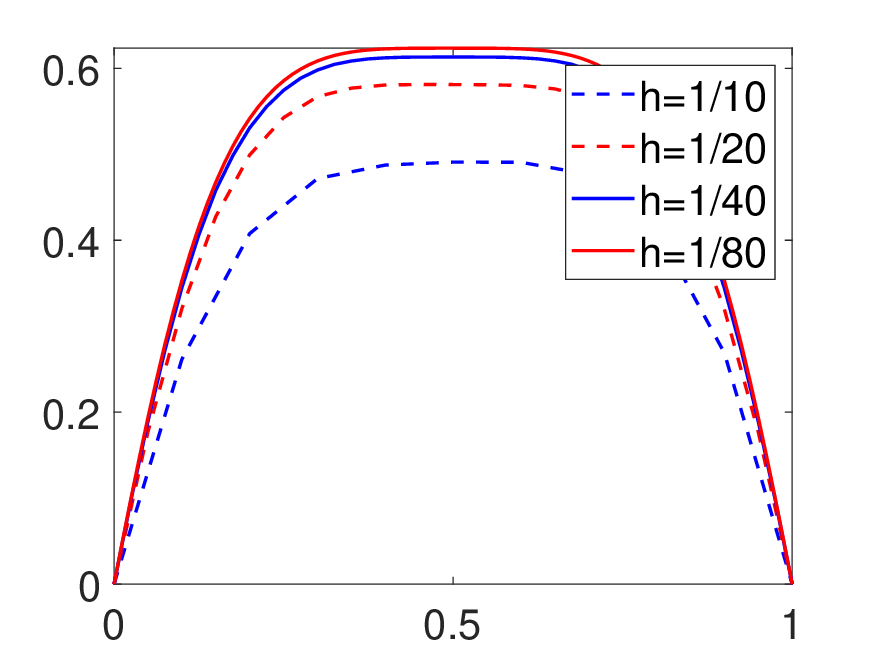} &
		\includegraphics[width=0.3\textwidth]{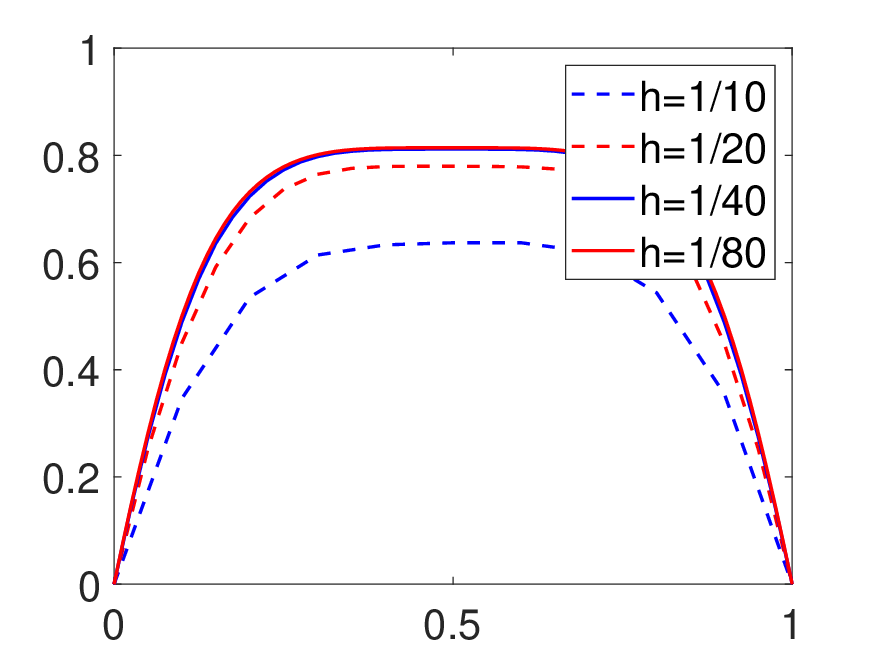}\\
		\multicolumn{3}{c}{} \\
		\includegraphics[width=0.3\textwidth]{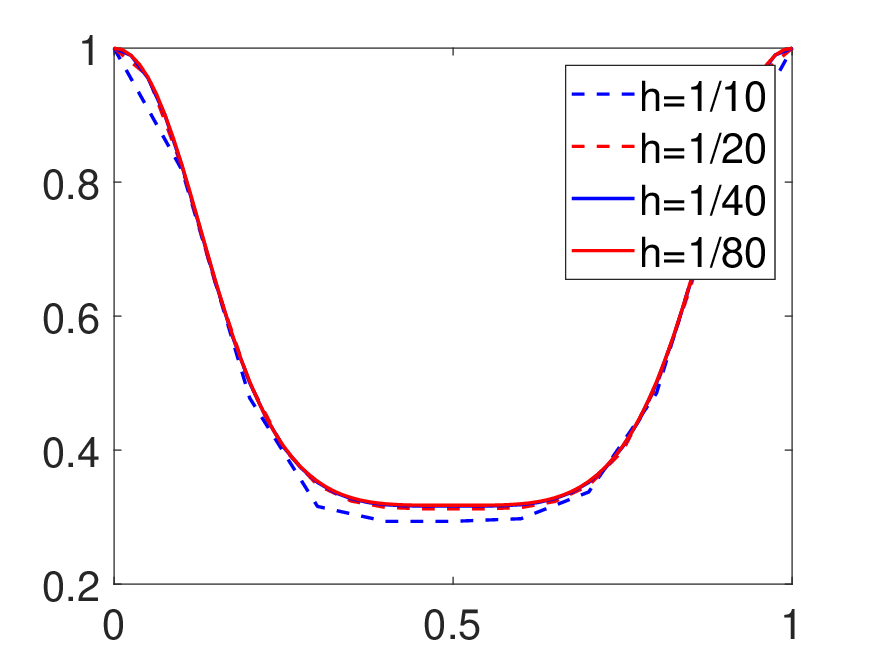} &
		\includegraphics[width=0.3\textwidth]{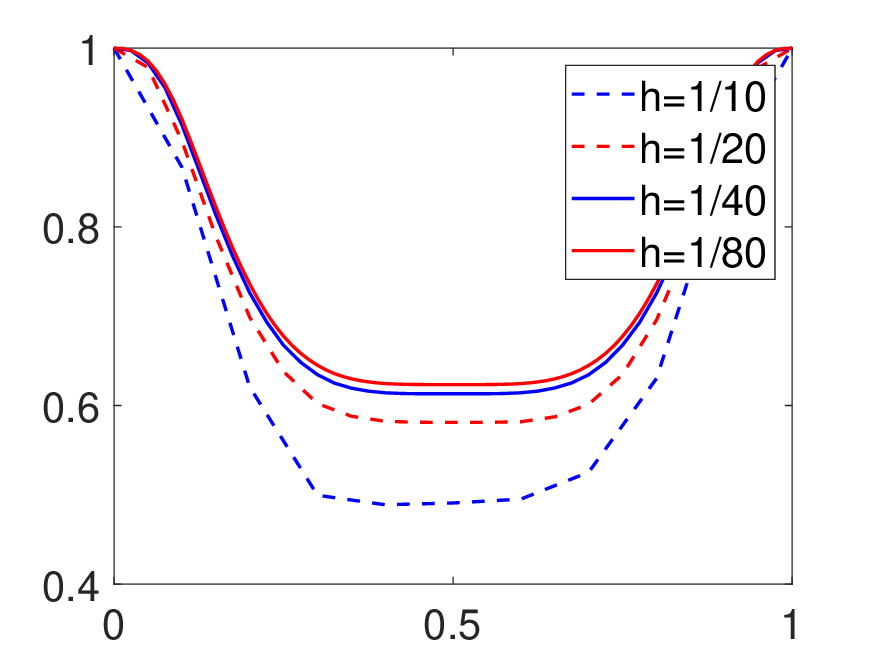} &
		\includegraphics[width=0.3\textwidth]{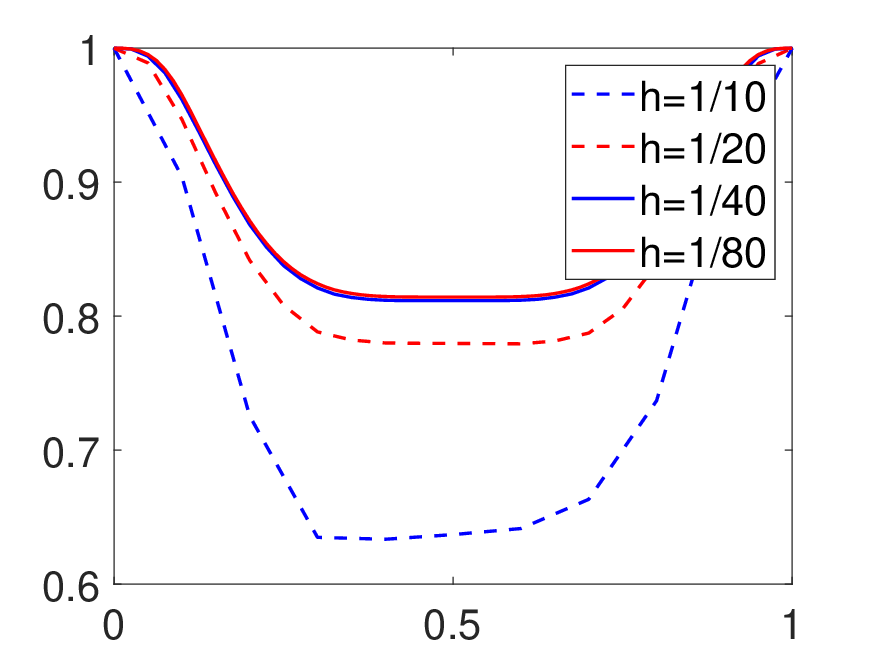} 
		\\
	\end{tabular}
	\caption{(Pucci's equation.) Numerical results for problem \eqref{P-nonsmooth} on the unstructured mesh, Figure \ref{mesh}(b): cross sections of the computed solution along the line $x_1 = 0.5$ (first row) and the line $x_1 = x_2$ (second row).}
	\label{P-nonsmooth-shoot}
\end{figure}

\begin{figure}[t!]
	\centering
	\begin{tabular}{c@{\hspace{1.5cm}}c}
		\includegraphics[width=0.35\textwidth]{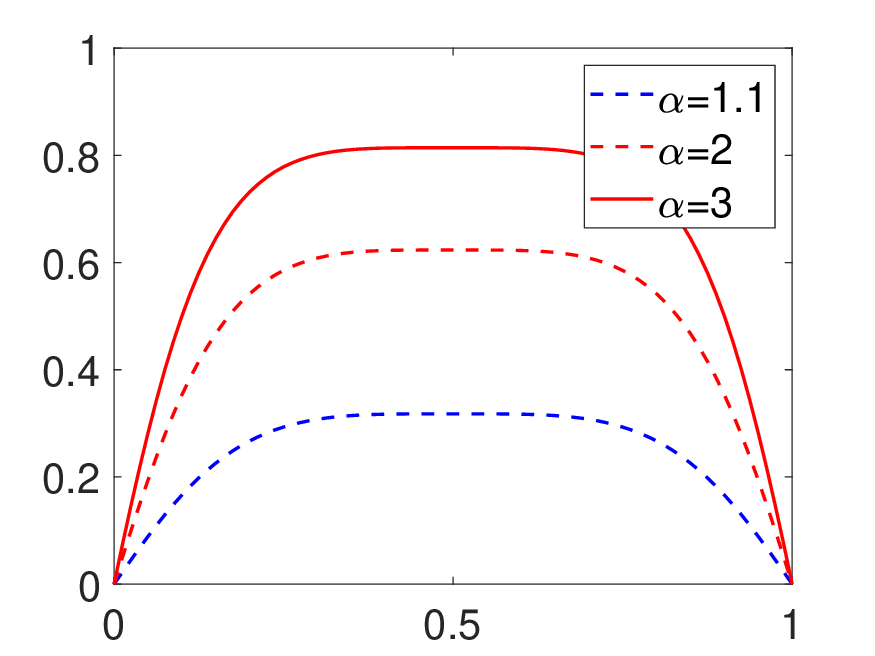} &
		\includegraphics[width=0.35\textwidth]{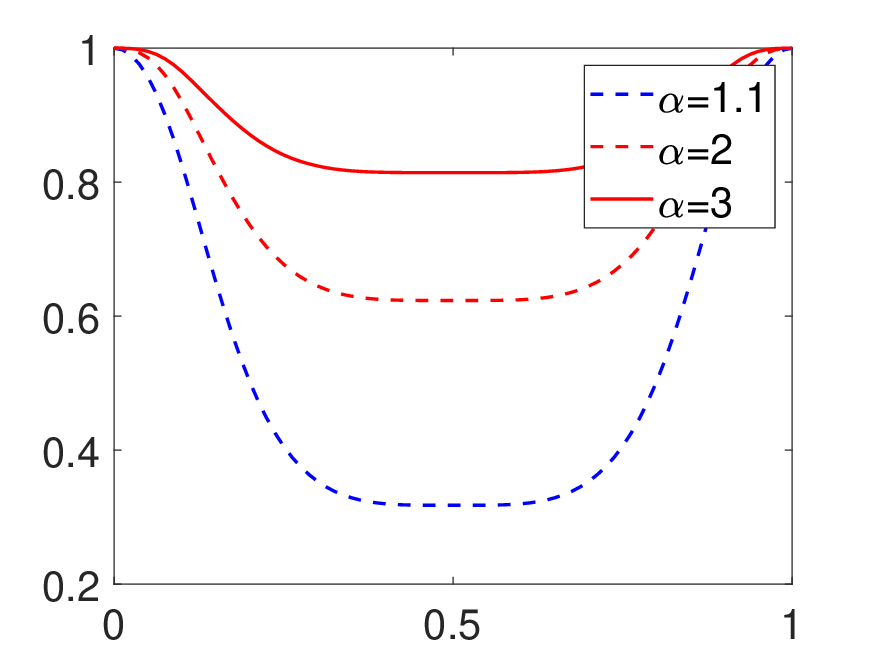}
	\end{tabular}
	\caption{(Pucci's equation.) Numerical results for problem \eqref{P-nonsmooth} on the unstructured mesh,  Figure \ref{mesh}(b), with $h=1/80$: cross sections of the computed solution along the line $x_1 = 0.5$ (left) and the line $x_1 = x_2$ (right) with $\alpha= 1.1$, $\alpha= 2.0$, and $\alpha = 3.0$, respectively.}
	\label{P-nonsmooth-shoot2}
\end{figure}

The related computational results are shown in Figures~\ref{P-nonsmooth-graph}, \ref{P-nonsmooth-shoot}, and \ref{P-nonsmooth-shoot2}. Figure~\ref{P-nonsmooth-graph} displays graphs of the numerical approximations for various values of \(\alpha\) with \(h = 1/80\). Figure~\ref{P-nonsmooth-shoot} shows cross sections of the computed solution along the line \(x_1 = 0.5\) (first row) and the line \(x_1 = x_2\) (second row) for \(\alpha = 1.1\) (left), \(\alpha = 2\) (middle), and \(\alpha = 3\) (right). We observe that our numerical solutions converge to a limit solution, which are consistent with those observed in \cite{Glowinski}.  Furthermore,
Figure~\ref{P-nonsmooth-shoot2} indicates that the value of \(u\) exhibits a positive correlation with the value of \(\alpha\).

\section{Conclusion}
\label{sec.conclusion}

We have developed a fast operator-splitting method for solving two-dimensional semilinear elliptic equations and Monge-Amp\`ere type equations. For the semilinear case, the convergence of the proposed method is established. To address Monge-Amp\`ere type equations, we employ a novel eigenvalue-based reformulation that transforms them into a semilinear form. The scheme is spatially discretized using a mixed finite element method with piecewise linear bases, making it straightforward to apply to problems on both polygonal and curved domains. Extensive numerical experiments show that our approach is more efficient than those existing methods while delivering comparable or superior accuracy. For the semilinear equation, the Dirichlet Monge-Amp\`ere equation, and the Pucci's equation, our method achieves the optimal convergence rate when a classical solution exists.

		\bibliographystyle{abbrv}
		\bibliography{references}

	\end{document}